\documentclass[letterpaper, 10 pt, conference]{ieeeconf}  

\IEEEoverridecommandlockouts

\usepackage[colorinlistoftodos,bordercolor=orange,backgroundcolor=orange!20,linecolor=orange,textsize=scriptsize]{todonotes}

\usepackage{amsmath,amsfonts,amssymb,graphicx,bm}
\usepackage{enumerate}
\usepackage{pdfsync,color} 
\usepackage{breqn}

\usepackage{bbm}

\newtheorem{theorem}{Theorem}
\newtheorem{definition}[theorem]{Definition}
\newtheorem{proposition}[theorem]{Proposition}
\newtheorem{corollary}[theorem]{Corollary}
\newtheorem{lemma}[theorem]{Lemma}
\newtheorem{remark}[theorem]{Remark}

\newcommand{\RR}{\mathbb{R}}

\newcommand{\eps}{\epsilon}
\newcommand{\Id}{\mathrm{id}}

\newcommand{\lip}{\mathrm{Lip}\,}

\newcommand{\co}{\overline{\mathrm{co}}\,}

\newcommand{\un}{\mathbf{1}}

\newcommand{\R}{{\mathbb {R}}}

\newcommand{\la}{\left\langle}
\newcommand{\ra}{\right\rangle}

\newcommand{\q}{q}
\newcommand{\M}{\mathcal M}
\renewcommand{\S}{\mathcal S}

\renewcommand{\l}{L}
\newcommand{\f}{\frac}
\newcommand{\lb}{\lambda}
\newcommand{\al}{\alpha}
\DeclareMathOperator*{\argmax}{arg\,max}

\newcommand{\beq}{\begin{equation}}
\newcommand{\eeq}{\end{equation}}

\newcommand{\vc}[1]{{#1}}

\makeatletter
\def\moverlay{\mathpalette\mov@rlay}
\def\mov@rlay#1#2{\leavevmode\vtop{%
   \baselineskip\z@skip \lineskiplimit-\maxdimen
   \ialign{\hfil$\m@th#1##$\hfil\cr#2\crcr}}}
\newcommand{\charfusion}[3][\mathord]{
    #1{\ifx#1\mathop\vphantom{#2}\fi
        \mathpalette\mov@rlay{#2\cr#3}
      }
    \ifx#1\mathop\expandafter\displaylimits\fi}
\makeatother

\newcommand{\cupdot}{\charfusion[\mathbin]{\cup}{\cdot}}

\title{Non-linear eigenvalue problems arising from growth maximization of positive linear dynamical systems}

\title{Non-linear eigenvalue problems arising from growth maximization of positive linear dynamical systems}

\author{Vincent Calvez$^1$ \ \and Pierre Gabriel$^2$ \ \and St\'ephane Gaubert$^3$
\thanks{$^{1}$ Ecole Normale Sup\'erieure de Lyon, 
UMR CNRS 5669 'UMPA', and project-team Inria NUMED, 46, all\'ee d'Italie 
69364 Lyon Cedex 07,
France.
E-mail: {\tt\footnotesize vincent.calvez@ens-lyon.fr}}%
\thanks{$^{2}$ Laboratoire de Math\'ematiques de Versailles
Universit\'e de Versailles St-Quentin-en-Yvelines, 
45 avenue des \'Etats-Unis, 78035 Versailles Cedex, France.
E-mail: {\tt\footnotesize pierre.gabriel@uvsq.fr }}%
\thanks{$^{3}$ Inria and CMAP, UMR CNRS 7641, Ecole Polytechnique, Route de Saclay, 91128 Palaiseau Cedex, France.
E-mail: {\tt\footnotesize stephane.gaubert@inria.fr}}%
}

\begin{document}

\maketitle

\begin{abstract}
We study a growth maximization problem for a continuous time positive
linear system with switches. This is motivated by a problem of
mathematical biology (modeling growth-fragmentation processes and the PMCA protocol).  We show that the growth rate is
determined by the non-linear eigenvalue of a max-plus analogue of the
Ruelle-Perron-Frobenius operator, or equivalently, by the ergodic
constant of a Hamilton-Jacobi (HJ) partial differential equation,
the  solutions or subsolutions of which yield Barabanov and extremal norms,
respectively. We
exploit contraction properties of order preserving flows, with respect to Hilbert's projective metric, to show that the non-linear eigenvector of the
operator, or the ``weak KAM'' solution of the HJ equation, does
exist. Low dimensional examples are presented, showing that the optimal control can lead to
a limit cycle.
\end{abstract}

\section{Introduction}



We investigate in this note the optimal control of time continuous positive linear dynamical systems in infinite horizon. 
We wish to compute the maximal growth rate
that can be obtained from infinitesimal combinations of a set of nonnegative matrices. 

More precisely, we consider a compact set $\M\subset \mathcal M_n(\R)$ 
of irreducible Metzler matrices. 
That is to say, we assume that for all $m\in \mathcal M$ and for all $i\neq j$,  $m_{ij}\geq 0$. 
In addition for every partition of indices $\{ 1\dots n \} = \vc{I \cupdot J}$ one can pick $i\in I$ and $j\in J$ such that $m_{ij} >0$. 
A direct consequence of compactness is {\it uniform irreducibility}: there exists a constant $\nu>0$ such that for all $m\in \M$, and every partition of indices one can pick $i\in I$ and $j\in J$ such that $m_{ij} \geq \nu$.

%
%
Let $K$ be the nonnegative orthant in $\R^n$, $K_+$ the positive orthant, and $K_0=K\setminus\{0\}$.
For $t>0$, $x\in K$ and a measurable control function $M:[0,t]\to \M$, we define $x_M\in W^{1,\infty}([0,t], \R^n)$ as the solution of the following linear problem with control $M$:
\begin{equation} \begin{cases}
\dot x_M(s) = M(s) x_M(s),\\
x_M(0) = x\,.
\end{cases}
\label{eq:ODE}
\end{equation}
We also denote $x_M(s) = R(s,M)x$, where $R$ is the resolvent. Finally \vc{we denote in short $L^\infty(0,t)$} the set of measurable (bounded by assumption) control functions $M:[0,t]\to \M$. We are interested in control functions maximizing the growth rate
\begin{align}\label{e-growth}
\limsup_{t\to \infty} \frac{1}{t}\log \|x_M(t)\| \enspace .
\end{align}

We assume w.l.o.g. that $\M$ is convex. The results presented here are still valid for nonconvex sets $\M$, provided the controls are replaced by  relaxed controls which take values in the closed convex hull $\co(\M)$.

For a constant control $M(s) \equiv m$ we have $R(t,m) = e^{tm}$. It is an immediate consequence of the Perron-Frobenius theorem that, being \vc{$\phi_m\in K_+$} a left Perron-Frobenius (PF) eigenvector of $m$, the linear function $\overline v(x) =  \la \phi_m,x\ra $ satisfies the following identity,
\[ (\forall t\in \RR_+)\; (\forall x\in K)\quad e^{\lambda(m) t}\overline v(x) =   \overline v(x_m(t))  \,,  \]
where $\lambda(m) \in \RR$ is the dominant eigenvalue of $m$.
The following result can be thought of as a non-linear extension of the Perron-Frobenius theorem. 

\begin{theorem}\label{thm:optimal growth}
Under previous assumptions there exist a real \vc{$\lambda(\mathcal M)$} and a function $\overline v: K\to \R_+$, homogeneous of degree 1, \vc{positive on $K_0$,} globally Lipschitz continuous,
which satisfy the following identity
\begin{equation}
(\forall t\in \RR_+)\; (\forall x\in K)\quad e^{\vc{\lambda(\M)} t} \overline{v}(x) =  \sup_{M\in L^\infty(0,t)}  \overline{v}(x_M(t)) \, ,
\label{eq:fixed point HJ}
\end{equation}
The scalar $\lambda(\mathcal M)$ is unique as soon as $\overline v$ belongs to the class of homogeneous functions of degree 1 which are locally bounded on $K$,
and it determines the optimal growth rate~\eqref{e-growth}.
Moreover, $\overline{u}=\log \overline{v}$ is 
characterized as a viscosity solution 
of an ergodic Hamilton-Jacobi PDE:
\begin{equation}
- \lambda(\mathcal M) +  H(D_y\overline u(y) , y ) = 0 \, , \vc{\quad y\in \S}\, ,
\label{eq:ergodic HJ:intro}
\end{equation}
where $\S$ is the standard simplex.
\end{theorem}
The Hamiltonian $H$ will be given in Section~\ref{sec-sketch}.

\begin{corollary}[Ergodicity] \label{cor:ergodicity}
Let $v_0: K\to \R_+$ be a continuous function, homogeneous of degree 1, positive on $K_0$. Define
$v(t,x) = \sup_{M\in L^\infty(0,t)} v_0(x_M(t))$.
Then we have the following ergodicity result,
\[ (\forall x\in \vc{K_0})\quad  \lim_{t\to +\infty} \frac1t  \log (v(t,x))  = \lambda(\M)\, . \ \]
Moreover the convergence is locally uniform on $K_0$.
\end{corollary}

Theorem~\ref{thm:optimal growth} is closely related to results belonging to the theory of stability of linear inclusions.
There, matrices are not necessarily assumed to be Metzler matrices.
The non-linear eigenvalue $\lambda({\mathcal M})$
coincides with the joint spectral radius \cite{rota_note_1960}. In his seminal paper \cite{barabanov_absolute_1988}, Barabanov proved the existence of extremal norms in $\R^n$ which saturates \eqref{eq:fixed point HJ},
under a different irreducibility condition.
Later the same author investigated the behaviour of extremal trajectories in the three-dimensional case $n=3$, first when $\M$ has the specific structure of a segment with a rank one matrice for direction \cite{barabanov_ui_1993}, secondly under a uniqueness condition for extremal trajectories verifying the Pontryagin Maximum Principle (PMP) \cite{barabanov_asymptotic_2008} (see also the recent improvement by Gaye et al \cite{gaye_properties_2013}). 
We also refer to \cite{wirth_generalized_2002} for an alternative proof of the existence of Barabanov extremal norms, and to the work of Chitour, Mason and Sigalotti~\cite{chitour} for the analysis of situations in which there are obstructions to the existence of such norms.

Several authors have analyzed specially the stability of positive linear systems. Very recently, Mason and Wirth~\cite{mason_extremal_2014} have established the existence of
an extremal norm, that is, a viscosity subsolution of the spectral problem~\eqref{eq:fixed point HJ} (the equality relation being replaced by $\geq$), corresponding
to a critical subsolution of the ergodic Hamilton-Jacobi equation.
They use an irreducibility condition which is milder than our,
but which does not guarantee the existence of a viscosity solution.
Conditions for the existence of subsolutions are typically
less restrictive. It is an interesting
issue to see whether the assumptions of Theorem~\ref{thm:optimal growth} could
be relaxed.

We emphasize that we take advantage of an illuminating connection between problem \eqref{eq:fixed point HJ} and the weak KAM theory in Lagrangian dynamics \cite{Fathi-book}. In particular long-time dynamics of optimal trajectories appear to be encoded in the so-called Aubry sets.
Such eigenproblems have
been widely studied in ergodic control, and also by dynamicians
in the setting of the weak KAM theory, where the eigenfunction
is known as a weak KAM solution. 
However, basic existence
results for eigenvectors 
rely on controllability conditions which are not satisfied
in our setting.

We exploit tools from the theory of Hamilton-Jacobi PDE to prove Theorem \ref{thm:optimal growth}, combined with techniques from Perron-Frobenius theory. In particular,
we use the Birkhoff-Hopf theorem in a crucial way.
The latter states that a linear map leaving invariant the
interior of a closed, convex and pointed cone is a strict contraction
in Hilbert's projective metric. The contraction of the controlled
flow turns out to entail the existence of the eigenvector. 
We note that tools from Lagrangian dynamics
(Mather sets) have been recently applied by Morris to study joint spectral radii~\cite{morris}. This deserves to be further studied in the present setting.

The same type of equations has been studied
in the context of infinite dimensional max-plus spectral theory.
In particular, the existence of continuous eigenfunctions
for max-plus operators with a continuous kernel is established
in~\cite{maslovkolokoltsov95}. More general conditions, exploiting
quasi-compactness techniques, can be found 
in~\cite{Nuss-Mallet}. It would be interesting to see
whether such techniques to the present problems.


A natural question that arises in the literature is whether the knowledge of $\{\lambda(m)\}_{m\in \M}$, say $(\forall m \in \M)\; \lambda(m)<0$ guarantees the stability of the differential inclusion \eqref{eq:ODE}. A positive answer has been given in \cite{gurvits_stability_2007} in dimension $n=2$. A negative answer has been given in (possibly) high dimension in the same work. Soon after, Fainshil et al give a counter-example in dimension $n=3$ \cite{fainshil_stability_2009}. It is a pair of matrices such that every convex combination has a negative spectral radius but the associated joint spectral radius is positive.

We address similar questions in the present note, namely whether $\lambda(\M) = \max_m  \lambda(m)$ or $\lambda(\M) > \max_m  \lambda(m)$. We give a new and self-contained proof of the positive answer in dimension $n=2$. We also give three dimensional numerical examples with positive and negative answers. The case where $\lambda(\M) > \max_m \lambda(m)$ is of particular interest. To find such a numerical example we restrict to the case where $\M$ is a segment, and the maximum of $\lambda(m)$ is attained at an interior point. We investigate  periodic perturbations of the optimal constant control in the spirit of \cite{Clairambault,Lepoutre}. More precisely we compute the second order directional derivative of the Floquet eigenvalue. We derive a criterion about the local optimality of the constant control with respect to periodic perturbations. We exhibit a numerical example for which this condition is satisfied. Numerical simulations of the full optimal control problem clearly shows the convergence of the optimal trajectory towards a limit cycle, suggesting that the optimal control in infinite horizon is indeed a BANG-BANG periodic control. It is worth noticing that the criterion that we derive is the exact opposite of a so-called Legendre condition in geometric optimal control theory  \cite{AgrachevSachkov,BonnardCaillauTrelat}. The latter condition ensures the local optimality of the extremal trajectory (here the trajectory corresponding to the maximal Perron eigenvalue) for short times.


\if{
In Section~\ref{sec-qualitative}, we investigate the qualitative properties of the nonlinear eigenvalue $\lambda(\mathcal M).$
First we prove that in dimension 2, this value corresponds to the maximal Perron eigenvalue among the matrices $m\in\mathcal M.$
Then we investigate in higher dimensions the possible situations where the optimal growth rate is not provided by a constant control. Using the idea of comparing Floquet and Perron eigenvalues, as it is done in [Clairambault-Gaubert-Perthame, Clairambault-Gaubert-Lepoutre] to compare the growth rates of cell populations in the context of cyrcadian rhythms, we make periodic variations of the control around the matrix with the optimal Perron eigenvalue. The computation of the directional second derivative of the Floquet eigenvalue makes appear a sufficient condition to ensure that $\lambda(\mathcal M)$ is strictly larger than the maximal Perron eigenvalue of $\mathcal M.$ We exhibit numerical examples in dimension 3 for which this condition is satisfied, leading to the convergence of the optimal trajectory toward a limit cycle.
Surprisingly enough, we remark that the condition on the second derivative of the Floquet eigenvalue is the exact opposite of a so-called Legendre condition in the geometric theory of optimal control (see~\cite{AgrachevSachkov,BonnardCaillauTrelat} for recent works on this theory). The latter condition ensures the local optimality of the best constant trajectory (i.e. the trajectory corresponding to the maximal Perron eigenvalue) for short times, while the former relies on the asymptotic behaviour for large time.}\fi

%
%
\if{
Let $K$ be the nonnegative orthant in $\R^n$, $K_+$ the positive orthant, and $K_0=K\setminus\{0\}$. 
For $t>0$, $x\in K$ and a measurable control function $M:[0,t]\to \M$, we define $x_M\in W^{1,\infty}([0,t], \R^n)$  the solution of the following linear problem with control $M$:
\begin{equation} \begin{cases}
\dot x_M(s) = M(s) x_M(s),\\
x_M(0) = x\,.
\end{cases}
\label{eq:ODE}
\end{equation}
We also denote $x_M(s) = R(s,M)x$, where $R$ is the resolvent. Finally \vc{we denote in short $L^\infty(0,t)$} the set of measurable (bounded by assumption) control functions $M:[0,t]\to \M$. 

We assume w.l.o.g. that $\M$ is convex. The results presented here are still valid for nonconvex sets $\M$, provided the controls are replaced by  relaxed controls which take values in the closed convex hull $\co(\M)$.

For a constant control $M(s) \equiv m$ we have $R(t,m) = e^{tm}$. It is an immediate consequence of the Perron-Frobenius theorem that, being \vc{$\phi_m\in K_+$} a left Perron-Frobenius (PF) eigenvector of $m$, the linear function $\overline v(x) =  \la \phi_m,x\ra $ satisfies the following identity,
\[ (\forall t\in \RR_+)\; (\forall x\in K)\quad e^{\lambda(m) t}\overline v(x) =   \overline v(x_m(t))  \,,  \]
where $\lambda(m) \in \RR$ is the dominant eigenvalue of $m$. 
The following result is a non-linear extension of the Perron-Frobenius theorem for optimal control in infinite horizon.

\begin{theorem}\label{thm:optimal growth}
Under previous assumptions there exist a real \vc{$\lambda(\mathcal M)$} and a function $\overline v: K\to \R_+$, homogeneous of degree 1, \vc{positive on $K_0$,} globally Lipschitz continuous,
which satisfies the following identity 
\begin{equation} 
(\forall t\in \RR_+)\; (\forall x\in K)\quad e^{\vc{\lambda(\M)} t} \overline{v}(x) =  \sup_{M\in L^\infty(0,t)}  \overline{v}(x_M(t)) \, , 
\label{eq:fixed point HJ}
\end{equation}
Moreover $\lambda(\mathcal M)$ is unique as soon as $\overline v$ belongs to the class of homogeneous functions of degree 1 which are locally bounded on $K$.
\end{theorem}

\begin{corollary}[Ergodicity] \label{cor:ergodicity}
Let $v_0: K\to \R_+$ be a continuous function, homogeneous of degree 1, positive on $K_0$. Define 
$v(t,x) = \sup_{M\in L^\infty(0,t)} v_0(x_M(t))$.
Then we have the following ergodicity result,
\[ (\forall x\in \vc{K_0})\quad  \lim_{t\to +\infty} \frac1t  \log (v(t,x))  = \lambda(\M)\, . \ \]
Moreover the convergence is locally uniform on $K_0$.
\end{corollary}

Theorem \ref{thm:optimal growth} is a special case of more general results stated in the theory of stability of linear inclusions. There matrices are not assumed to be Metzler matrices. The non-linear eigenvalue coincides with the joint spectral radius \cite{Rota-Strang}. In its seminal paper \cite{Barabanov}, Barabanov proved the existence of extremal norms in $\R^n$ which saturates \eqref{eq:fixed point HJ}. Later the same author investigated the behaviour of extremal trajectories in the three-dimensional case $n=3$, first when $\M$ has the specific structure of a segment with a rank one matrice for direction \cite{barabanov_1993}, secondly under a uniqueness condition for extremal trajectories verifying the Pontryagin Maximum Principle (PMP) \cite{barabanov_2008} (see also the recent improvement by Gaye et al \cite{gaye_2014}). We also refer to \cite{Wirth} for an alternative proof of the existence of Barabanov extremal norms. In parallel, several authors have analyzed stability of positive linear systems (see \cite{Mason-Wirth} and references therein). A natural question that arises in the literature is whether the knowledge of $\{\lambda(m)\}_{m\in \M}$, say $(\forall m \in \M)\; \lambda(m)<0$ guarantees the stability of the differential inclusion \eqref{eq:ODE}. A positive answer has been given in \cite{gurvits_shorten_mason_2007} in  dimension $n=2$. A negative answer has been given in (possibly) high dimension in the same work. Soon after, Fainshil et al give a counter-example in dimension $n=3$. It is a pair of matrices such that every convex combination has a negative spectral radius but the associated joint spectral radius is positive.  

We address similar questions in the present note, namely whether $\lambda(\M) = \max_m  \lambda(m)$ or $\lambda(\M) > \max_m  \lambda(m)$. We give a new and self-contained proof of the positive answer in dimension $n=2$. We also give three dimensional numerical examples with positive and negative answers. The case where $\lambda(\M) > \max_m  \lambda(m)$ is of particular interest. To find such a numerical example we restrict to the case where $\M$ is a segment, and the maximum of $\lambda(m)$ is attained at an interior point. We investigate  periodic perturbations of the optimal constant control in the spirit of \cite{Clairambault-Gaubert-Perthame, Clairambault-Gaubert-Lepoutre}. More precisely we compute the second order directional derivative of the Floquet eigenvalue. We derive a criterion about the local optimality of the constant control with respect to periodic perturbations. We exhibit a numerical example in for which this condition is satisfied. Numerical simulations of the full optimal control clearly shows the convergence of the optimal trajectory toward a limit cycle, suggesting that in this case the optimal control in infinite horizon is indeed a BANG-BANG periodic control. It is worth noticing that the criterion that we derive is the exact opposite of a so-called Legendre condition in geometric optimal control theory  \cite{AgrachevSachkov,BonnardCaillauTrelat}. The latter condition ensures the local optimality of the best constant trajectory (here the trajectory corresponding to the maximal Perron eigenvalue) for short times.

We emphasize that we develop here a new approach which takes advantage of an illuminating connection between problem \eqref{eq:fixed point HJ} and the weak KAM theory in Lagrangian dynamics \cite{fathi_book}. We apply tools from Hamilton-Jacobi PDE to prove Theorem \ref{thm:optimal growth}. We postulate that long-time dynamics of optimal trajectories are encoded in the so-called Aubry sets. In particular we give an example where the Aubry set is a limit cycle.  

\section{Non-linear eigenvectors}
Let $\M$ be a compact subset of $\mathcal P\subset \mathcal M_n(\R)$ the set of matrices which are both nonnegative off-the-diagonal and irreducible. A direct consequence of compactness is {\it uniform irreducibility}: it exists a constant $\nu>0$ such that for all $m\in \M$, and every partition of indices $\{ 1\dots n \} = \vc{I \cupdot J}$ one can find $i\in I$ and $j\in J$ such that $m_{ij} \geq \nu$. 

%
%
Let $K$ be the nonnegative orthant in $\R^n$, $K_+$ the positive orthant, and $K_0=K\setminus\{0\}$. 
For $t>0$, $x\in K$ and a measurable control function $M:[0,t]\to \M$, we define $x_M\in W^{1,\infty}([0,t], \R^n)$  the solution of the following linear problem with control $M$:
\begin{equation} \begin{cases}
\dot x_M(s) = M(s) x_M(s),\\
x_M(0) = x\,.
\end{cases}
\label{eq:ODE}
\end{equation}
We also denote $x_M(s) = R(s,M)x$, where $R$ is the resolvent of the flow. Finally \vc{we denote in short $L^\infty(0,t)$} the set of measurable (bounded) control functions $M:[0,t]\to \M$. 

\vc{We make the simplifying assumption that $\M$ is convex for the sake of presentation. Results contained in this work are still valid for nonconvex sets $\M$, provided the optimal controls are replaced with optimal relaxed controls which take values in the closed convex hull $\co(\M)$.}

For a constant control $M(s) \equiv m$ we have $R(t,m) = e^{tm}$. It is an immediate consequence of the Perron-Frobenius Theorem that, being \vc{$\phi_m\in K_+$} a left Perron eigenvector of $m$, the linear function $\overline v(x) =  \la \phi_m,x\ra $ satisfies the following identity,
\[ (\forall t\in \RR_+)\; (\forall x\in K)\quad e^{\lambda(m) t}\overline v(x) =   \overline v(x_m(t))  \,,  \]
where $\lambda(m) \in \RR$ is the dominant eigenvalue of $m$. 

The following result may be thought of as an extension
of the Perron-Frobenius theorem to the present
growth maximization problem.

\begin{theorem}\label{thm:optimal growth}
Under previous assumptions there exist a real \vc{$\lambda(\mathcal M)$} and a function $\overline v: K\to \R_+$, homogeneous of degree 1, \vc{positive on $K_0$,} Lipschitz continuous,
satisfying the following identity 
\begin{equation} 
\!(\forall t\in \RR_+)\; (\forall x\in K)\quad e^{\vc{\lambda(\M)} t} \overline{v}(x) = \!\!\! \!\!\!\sup_{M\in L^\infty(0,t)} \!\!\! \!\!\overline{v}(x_M(t)) , 
\label{eq:fixed point HJ}
\end{equation}
Moreover the "eigenvalue" \vc{$\lambda(\mathcal M)$} is unique with respect to the class of homogeneous functions of degree 1 which are locally bounded on $K$.
\end{theorem}

\begin{corollary}[Ergodicity] \label{cor:ergodicity}
Let $v_0: K\to \R_+$ be a continuous function, homogeneous of degree 1, positive on $K_0$. Define 
$v(t,x) = \sup_{M\in L^\infty(0,t)} v_0(x_M(t))$.
Then we have the following ergodicity result,
\[ (\forall x\in \vc{K_0})\quad  \lim_{t\to +\infty} \frac1t  \log (v(t,x))  = \lambda(\M)\, . \ \]
Moreover the convergence is locally uniform on $K_0$.
\end{corollary}

}\fi
\section{Techniques of proof of Theorem \ref{thm:optimal growth}}\label{sec-sketch}
We present in this section the main elements of the proof of Theorem~\ref{thm:optimal growth}. In this Section we write in short $\lambda = \lambda(\M)$.

%
\medskip
\noindent{\bf Step \#1.} {\bf Homogeneity and projection of the dynamics onto the simplex.} 
The infinitesimal version of \eqref{eq:fixed point HJ} writes as a Hamilton-Jacobi equation in the viscosity sense, 
\begin{equation} \lambda \overline v(x) = \max_{m\in \mathcal M} \la D_x \overline v(x),m x \ra\, . \label{eq:infinitesimal} \end{equation}
Using the homogeneity of the function $\overline v$ we can project \eqref{eq:infinitesimal} onto the simplex $\S = \{x\in K: \langle\un,x\rangle = 1\}$. We write 
\[
\overline v(x) = \langle\un,x\rangle \mathring v\left(\frac x{\langle\un,x\rangle}\right)\, ,
\]
where $\mathring v$ is defined on $\S$. Then problem \eqref{eq:infinitesimal} is equivalent to finding $(\lambda, \mathring v)$ such that 
\begin{equation} 
\lambda \mathring v(y) = \max_{m\in \M} \left(  \l(y,m) \mathring v(y) + \la D_y \mathring v(y) , b(y,m)\ra \right)\, , \label{eq:HJ v} \end{equation}
where the pay-off $\l$ and the vector fields $ b$ are given by
\[ \l(y,m) = \la \un , m y\ra \, , \quad b(y,m) = m y - \l(y,m) y\, . \]
Note that each vector field $b(\cdot,m)$ is tangent to the simplex $\S$. It gives indeed the projected dynamics on  the simplex: if $x_M$ is solution to \eqref{eq:ODE} then $y_M = \frac{x_M}{\langle\un,x_M\rangle}$ is solution to the non-linear ODE
\begin{equation}\label{eq:simplex dyn}\dot y_M(s) = b(y_M(s),M(s))\,.\end{equation} 

\medskip

\noindent{\bf Step \#2.} {\bf Computation of a Lipschitz constant with respect to Hilbert's projective metric}
Recall that {\em Hilbert's (projective) metric}
is defined, for all $x,y\in K_+$,
by \begin{align}
d(x,y) = \log \max_{1\leq i,j\leq n} \frac{x_iy_j}{x_jy_i} \enspace .
\label{e-def-hilbert}
\end{align}
It is a metric in the set of half-lines included in the interior
of $K$. In particular, $d(x,y)=0$ iff $x$ and $y$ are proportional.
It is known to be a weak Finsler structure~\cite{nussbaum94}, obtained
by thinking of $K_+$ as a manifold with the
seminorm $\|h\|_x= \max_i h_ix_i^{-1} - \min _j h_j x_j^{-1}$ in the
tangent space at point $x$. Then, 
\[
d(x,y) = \inf \int_0^1 \|\dot{\gamma}(s)\|_{\gamma(s)} ds 
\]
where the infimum is taken over all differentiable paths $\gamma$ contained
in the interior of $K$, such that $\gamma(0)=x$ and $\gamma(1)=y$. 

Let $\q\in K_+$ and $m\in \mathcal P$. We aim to compute the Lipschitz constant of the function $l(x) = \frac{\la \q,mx\ra}{\la \q,x\ra}$, when the source set is endowed with the Hilbert metric. \vc{For a given matrix $m = (m_{ij})$ we denote $|m| = (|m_{ij}|)$ the matrix obtained by taking absolute values of the coefficient pointwise.} 

The following lemma is established by exploiting the Finsler's nature of Hilbert projective metric, along the lines of~\cite{nussbaum94}. See also~\cite{zhengstephane13}.
\begin{lemma}\label{lem:lip}
Let $l: (K,d)\to (\R,|\cdot|)$ defined as $l(x) = \frac{\la \q,mx\ra}{\la \q,x\ra}$. It is Lipschitz continuous with the following bound on the Lipschitz constant,
\[ \lip  l  \leq \sup_{x\in \vc{K_0}} \inf_{a\in \R}\dfrac{\la \q,\left|m  - a. \Id  \right|x \ra}{\la \q, x\ra}\, . \] 
\end{lemma}
\if{

\begin{proof}[Proof of Lemma \ref{lem:lip}]
The metric space $(K,d)$ possesses a Finsler structure (Nussbaum 1994). The norm on the tangent space is given by 
\[ \|v\|_x = \max_j \left(\dfrac {v_j} {x_j}\right) - \min_k \left(\dfrac {v_k} {x_k}\right) \, .\] 
The Lipschitz constant of $l$ on $K$ is given by
\[ \lip  l = \sup_{\vc{\begin{subarray}{c}(x,v)\in K\times \R^n \\ \|v\|_x\neq 0\end{subarray}}} \dfrac{|Dl(x) v|}{\|v\|_x}\,. \]
We write $v_i = u_i x_i$, and $m  = (m_{ij})$. 
We have for all $a\in \R$,
\begin{align*}  
&|Dl(x) v|  = \\
&\dfrac{\left|\left(\sum \q_im_{ij} u_j x_j\right)\left(\sum \q_k x_k\right) - \left(\sum \q_im_{ij} x_j\right)\left(\sum \q_k u_k x_k\right)\right|}{ \la \q,x\ra^2}\\
& = \dfrac{\left| \sum \q_i\q_k\left(m_{ij} - a \delta_{ij}\right) \left(u_j - u_k\right) x_j x_k\right|}{ \la \q,x\ra^2} \\
& \leq \dfrac{\la \q,\left|m - a.\Id \right|x \ra}{\la \q,x\ra} \left(\max u_j - \min u_k\right)
\end{align*}
This concludes the proof of Lemma \ref{lem:lip}. 
\end{proof}}\fi

\medskip

\noindent{\bf Step \#3.} {\bf Exponential contraction of the flow (after some time).}
A key technical ingredient is the following lemma, which shows that
for a fixed time $\tau>0$, the flow maps the closed cone to its interior.
\begin{lemma}\label{lem:stable cone}
Let $\tau>0$. Define the cone $K_\tau\subset K$ as the convex closure of images of $K$ by the flow after a time step $\tau>0$,
\[ K_\tau = \co\left(\bigcup_{M\in \mathcal{C}(\tau)} R(\tau,M)K  \right)\, .\]
It satisfies the following properties, 
\begin{itemize}
\item $K_\tau$ is stable with respect to every flow $R(s,M)$, $s\geq 0$, $M \in L^\infty(0,s)$. 
\item \vc{$K_\tau$ is included in the interior of the cone, closed, and bounded in Hilbert's projective metric.}
\end{itemize}
\end{lemma}
\if{
\begin{proof}
The cone $K_\tau$ is clearly \vc{positively} invariant by the flow $(R(s,M),s\geq0)$ for any control function $M$. Indeed by Caratheodory's Theorem any $z\in K_\tau$ is a linear convex combination of at most $n+1$ points in $\bigcup_{M\in \mathcal{C}(\tau)} R(\tau,M)K$: $z = \sum_i c_i R(\tau,M_i) x_i$. By linearity of the resolvent we have $R(s,N) z = \sum_i c_i R(\tau+s,M_i\sqcup N) x_i$, where $M_i\sqcup N$ denotes the concatenation of $M_i$ then $N$.  

We denote $\S_+ = \S\cap K_+$ the interior of $\S$. 
We shall ensure that $\S_\tau = \S\cap K_\tau $ is a compact subset of $\S_+$, in particular it does not touch the boundary $\partial \S$. It is sufficient to prove that $\bigcup_{M\in \mathcal{C}(\tau)} R(\tau,M)K $ satisfies this property. This is a consequence of uniform irreducibility of $\M$. 

We assume by contradiction that there are sequences of initial conditions $x_n$ and control functions $M_n$ such that $y_n(\tau) =  \frac{R(\tau,M_n)x_n}{\la \un, R(\tau,M_n)x_n\ra}$ converges towards $y_\tau\in   \partial \S$. First we observe that both the initial condition $y_n$ and the derivative $\dot y_n(s)$ are uniformly bounded by \eqref{eq:simplex dyn}. From Ascoli's Theorem we obtain uniform convergence of $y_{n}(s)$ towards a Lipschitz function $y(s)$ for $s\in [0,\tau]$, , up to extraction of a subsequence. Second we can extract another subsequence such that $M_n(s)$ converges weakly-$*$, towards a measurable function $M(s)\in L^\infty([0,\tau])$. The weak limit is uniformly irreducible by convexity of $\mathcal P$ and compactness of $\M$. Finally, the function $y_n(s)$ converges uniformly towards $y(s)$ which satisfies
\begin{equation} 
y(s) = y + \int_0^s \left(M(s')y(s') -\la \un,M(s')y(s')) \ra y(s')\right)\, ds'\, ,
\label{eq:strong limit}
\end{equation} 
for all $s\in [0,\tau]$.

Since $y(\tau)\in \partial \S$ there is an index $i$ such that $y_i(\tau) = 0$. We introduce $I$ the set of such indices. We deduce from the properties of the ODE system \eqref{eq:strong limit} and matrices $M(s)$ (namely nonnegative off-the-diagonal) that $(\forall i\in I)(\forall s \in [0,\tau])\; y_i(s) = 0$. By uniform irreducibility we can find a positive transition rate from the set $J = I^c$ to $I$, which contradicts the identity
\[ (\forall s\in [0,\tau]) \quad 0 = y_i(s) = \int_0^s  \sum_{j\in J} M_{ij}(s')y_j(s') \, ds'\, . \]  
This concludes the proof of  Lemma \ref{lem:stable cone}. 
\end{proof}
}\fi
%

A classical result of Birkhoff and Hopf shows that a linear map sending a (closed, convex, and pointed) cone to its interior is a strict contraction in Hilbert's projective metric, see for instance~\cite{lemmensnussbaum} for more information. We deduce from the Birkhoff-Hopf theorem and from Lemma~\ref{lem:stable cone} the following contraction result for the flow.
\begin{lemma}\label{lem:contraction}
There exist a time $T>0$ and a positive rate $\mu>0$  such that the flow $R(t,M)$ is uniformly exponentially contractive for $t\geq T$:
\begin{align}
(\forall t\geq  T)\; (\forall M \in L^\infty(0,t))\; (\forall (x,y)\in K_+\times K_+) \quad \nonumber\\ 
d(R(t,M)x, R(t,M)y) \leq e^{-\mu t} d(x,y)\, .
\end{align}
\end{lemma}
\if{
\begin{proof}
We introduce $k(t)$ the best contraction constant with respect to the Hilbert distance,
\begin{align*}
k(t) &= 
\sup  \log\left( \dfrac{d(R(t,M) x,R(t,M)y)}{d(x,y)}\right) ;\\
&
 (x,y)\in K_+\times K_+\, ,\; x\neq y \, , \; M \in L^\infty(0,t) \enspace .
 \end{align*}
The flows $R(t,M)$ are nonexpansive, thus $k(t)$ is nonpositive. In addition, Lemma \ref{lem:stable cone} for $\tau = 1$ guarantees that for all $M \in L^\infty(0,1)$,  $R(1,M)K_+  \subset K_1$, where  $K_1$ is compact for the topology induced by the Hilbert metric. Therefore, diameters are comparable: $\Delta(R(1,M))\leq \Delta(K_1) := \sup\{ d(x,y): (x,y)\in K_1\times K_1 \}<+\infty$. Birkhoff's theorem asserts that 
\[ k(1) \leq \log\left( \tanh\left(\frac{\Delta(K_1)}4 \right)\right) <0\, . \]
On the other hand, we deduce from a classical semi-group argument that the function $k(t)$ is subadditive [Hille-Phillips]. Therefore 
\[ \lim_{t\to +\infty} \dfrac{k(t)}{t} = \inf_{t>0} \dfrac{k(t)}{t}\leq k(1) < 0\, .  \]
As a consequence, there exists $T>0$ and $\mu>0$ such that $k(t)\leq - \mu t$ for $t\geq T$. 
\end{proof}
}\fi
\begin{remark}
If
$\inf_{m\in \M} \min_{i\neq j} m_{ij} >0$, one can choose $T=0$ in the Lemma \ref{lem:contraction}, and accordingly, 
\begin{equation}
\mu  = \inf_{m\in\M} \left(\min_{i\neq j} \left(  2 (m_{ij} m_{ji} )^{1/2}\right)\right)>0\, .
\label{eq:hyp mu0}
\end{equation}
See also~\cite{zhengstephane13}.
\end{remark}

\medskip

\noindent{\bf Step \#4.} {\bf Weak KAM Theorem.}
As suggested by the expected exponential growth, we make a logarithmic transformation. Let introduce $\mathring u =  \log \mathring v$. The original problem \eqref{eq:fixed point HJ} writes equivalently: find a real $\lambda$ and a function $\mathring u$, defined on the simplex $\S$, 
such that 
\begin{equation} \lambda t + \mathring u(y) \!= \!\!\sup_{M\in L^\infty(0,t)}\!\left\{ \int_0^t \l(y_M(s),M( s ))\, ds + \mathring u(y_M(t))\right\}\, , \label{eq:fixed point}  \end{equation}
for all $t\geq 0$, 
or in its infinitesimal setting: find a real $\lambda$ and a function $\mathring u$ such that $\mathring u$ is the viscosity solution of the stationary Hamilton-Jacobi equation 
\begin{equation}
- \lambda +  H(D_y\mathring u(y) , y ) = 0 \, , \vc{\quad y\in \S}\, ,
\label{eq:ergodic HJ}
\end{equation}
where the Hamiltonian is defined as $H(p,y) = \max_m\left(  \l (y,m) + \la p,b(y,m)  \ra \right)$.

The existence of a solution $(\lambda,u)$ is known
as a weak KAM Theorem in the context of dynamical systems, see the work by Fathi~\cite{Fathi,Fathi-book}.
Here, we follow the now classical  argument of Lions-Papanicolaou-Varadhan to prove the existence of such a pair $(\lambda,\mathring u)$, the vector $\mathring u$ 
being obtained as a rescaled limit of the solution $u_\epsilon$ of a
Hamilton-Jacobi PDE with discount rate $\epsilon>0$.
In doing so, we make use
of the contraction property of Lemma~\ref{lem:contraction}
with respect to Hilbert's projective metric.

\if{
%
We introduce some discount factor $\eps>0$. We define the function \vc{$u_\eps: \S\to \R$} as
\[u_\eps(y) = \sup_{M\in L^\infty(0,\infty)} \left\{  \int_0^\infty e^{-\eps t}\l (y_M(t), M(t))\, dt \right\}\, . \]
Classical dynamic programming principle states that $u_\eps$ is the viscosity solution of the stationary Hamilton-Jacobi-Bellman equation (Bardi, Capuzzo-Dolcetta),
\[ - \eps u_\eps(y) + H(D  u_\eps(y),y) = 0\, , \vc{\quad y\in \S}\, . \]

First we notice that $\eps u_\eps$ is bounded, with $\|\eps u_\eps\|_{L^\infty} \leq \|\l\|_{L^\infty}$. Next we prove that \vc{the restriction $u_\eps: \S_+\to \R$} is Lipschitz continuous \vc{with respect to the Hilbert metric}, uniformly with respect to $\eps>0$. We choose a close-to-optimal control function $M:[0,\infty)\to \mathcal{M}$, associated with the initial point \vc{$x\in \S_+$}. We choose the same control function for the initial point \vc{$y\in \S_+$},
\begin{dmath*}
u_\eps(y) - u_\eps(x) 
\leq \delta + \int_0^\infty e^{-\eps t}\left(  \l(y_M(t), M(t)) -\l(x_M(t), M(t))\right)\, dt \\
 \leq \delta +  (\sup_m \lip_x \l) \int_0^\infty e^{-\eps t}d(y_M(t),x_M(t)) \, dt \\
 \leq \delta +  (\sup_m \lip_x \l) \left(  \int_0^T e^{-\eps t}d(y_M(t),x_M(t)) \, dt + \int_T^\infty e^{-\eps t}d(y_M(t),x_M(t)) \, dt \right) \\
 \leq \delta +  (\sup_m \lip_x \l) \left(  \int_0^T e^{-\eps t}d(y,x ) \, dt + \int_T^\infty e^{-(\eps + \mu ) t}d(y,x ) \, dt \right) \\
 \leq \delta + (\sup_m \lip_x \l) \left(   T  + \dfrac{e^{-\mu T}}{ \mu } \right) d(y,x)\, .
\end{dmath*}
As $\delta$ is arbitrarily small, we conclude that $\lip u_\eps \leq \frac{(\sup_m \lip_x \l)}{\mu}$. Recall from \vc{Lemma~\ref{lem:lip}} that 
\[ \sup_m \lip_x\l \leq \sup_m \left( \sup_{x\in \vc{K_0}}\dfrac{\la \un,|m| x \ra}{\langle\un,x\rangle}\right)\leq \sup_m \max_{i,j} |m_{ij}|\, . \]
Therefore (up to extracting a subsequence) the function $\eps u_\eps$ converges towards a constant $\lambda$, locally uniformly over $\S_+$ (recall that the Hilbert metric is singular at the boundary $\partial \S$). Moreover the function $u_\eps - \min  u_\eps$ converges towards a Lipschitz continuous function \vc{$\mathring{u}:\S_+\to \R$} such that $\mathring{u}$ is the viscosity solution of the Hamilton-Jacobi equation \eqref{eq:ergodic HJ}, \vc{and satisfies the fixed point equation \eqref{eq:fixed point}}.

We can extend the uniform convergence of $\eps u_\eps$ 
up to the boundary $\partial \S$ by weakening the previous argument. We write
\begin{dmath*}
u_\eps(y) - u_\eps(x) 
 \leq \delta + \int_0^\infty e^{-\eps t}\left(  \l(y_M(t), M(t)) -\l(x_M(t), M(t))\right)\, dt \\
 \leq \delta + 2\|\l\|_\infty T +  (\sup_m \lip_x \l) \int_T^\infty e^{-\eps t}d(y_M(t),x_M(t)) \, dt \\
 \leq \delta + 2\|\l\|_\infty T + (\sup_m \lip_x \l) \dfrac{e^{-\mu T}}{ \mu } \Delta(K_T)\, .
\end{dmath*}
We deduce the uniform convergence of $\eps u_\eps$ \vc{towards $\lambda$} on the whole simplex $\S$.

}\fi

\medskip

\noindent{\bf Step \#5.} {\bf Calibrated trajectories.}
Before we proceed with the end of the proof (boundedness of $\mathring u$ and uniqueness of $\lambda$), we recall some definitions from \cite{Fathi-book} adapted to our context. 

\begin{definition}[Calibrated trajectories]
A Lipschitz curve $\gamma:I\to \vc{\S_+}$ defined on the interval $I\subset \R$, associated to some control $M\in L^\infty(I)$, $\gamma = y_M$, is calibrated if for every $t\leq t'\in I$, we have
\[ \mathring u(\gamma(t')) - \mathring u(\gamma(t)) = \int_t^{t'} \left( \l(y_M(s),M( s )) - \lambda\right)\, ds    \]
\end{definition}

\if{
\vc{
\begin{proposition}[Existence of calibrated curves]
\end{proposition}
}}\fi

Along the lines of~\cite{Fathi-book}, we show that calibrated
trajectories do exist. 

\medskip

\noindent{\bf Step \#6.} {\bf Regularity of $\mathring u$ up to the boundary $\partial \S$ and uniqueness of $\lambda$.}
First of all we deduce from the fixed point formulation \eqref{eq:fixed point} that $\mathring u$ is Lipschitz continuous on the whole $\S$ with the respect to the \vc{$\ell^1$} norm $|\cdot |_1$. Notice that the previous argument only yields local Lipschitz continuity due to the singularity \vc{of  the Hilbert metric at the boundary $\partial \S$}. From the fixed point formulation \eqref{eq:fixed point} we have in particular,
\begin{equation}\label{eq:fixed point T=1} \lambda 
+ \mathring u(y) = \!\!\!\!\sup_{M\in L^\infty(0,1)}\left\{ \int_0^1 \l(y_M(s),M( s ))\, ds + \mathring u(y_M(1))\right\}. \end{equation}
It suffices to observe that for all $M\in L^\infty(0,1)$, $y_M(1) = R(1,M)y\in K_1$  which is a compact subset of $\S$ with respect to the Hilbert metric. Thus $K_1$ is at uniform positive distance from the boundary $\partial \S$ and there exists a constant $C(K_1)$ such that for all $(x,y)\in K_1\times K_1$, $d(x,y) \leq C(K_1) |x-y|_1$. Finally we observe that \eqref{eq:fixed point T=1} is a supremum of Lipschitz functions as it is the case for $\mathring u(y_M(1))$:
\begin{dmath*} 
\left|\mathring u(y_M(1)) - \mathring u(x_M(1))\right| 
\leq \left(\lip \mathring u{|_{K_1}}\right) d (R(1,M)y,R(1,M)x) \\
\leq  \left(\lip \mathring u{|_{K_1}}\right) C(K_1) | R(1,M)y - R(1,M)x|_1 \\
 \leq \left(\lip \mathring u{|_{K_1}}\right) C(K_1) \left(\sup_{M\in L^\infty(0,1)} \|R(1,M)\|_1\right) |y-x|_1\, . 
\end{dmath*}
Therefore $\mathring u$ is globally Lipschitz on $\S$ with respect to the \vc{$\ell^1$} norm $|\cdot |_1$. \vc{As a consequence we can uniquely extend $\mathring u$ to a continuous function defined on $\S$.}
 

The uniqueness of $\lambda$ is then deduced from a classical argument, that
we skip, as well as the proof of Corollary~\ref{cor:ergodicity}.

\if{We repeat it here for the sake of completeness. 
Let $\lambda_1,\lambda_2$ associated to two bounded functions $\mathring u_1, \mathring u_2$ solutions of \eqref{eq:fixed point}. Take a calibrated curve associated to $\mathring u_1$ and some $y_0\in \S$ up to a time $T$ arbitrarily large. Then 
\begin{dmath*}
\lambda_1 T + \mathring u_1(y_0) = \int_0^T  \l(y_M(t),M(t ))\, dt + \mathring u_1(y_M(T)) \leq \lambda_2 T + \mathring u_2(y_2) - \mathring u_2(y_M(T)) + \mathring u_1(y_M(T))\, .
\end{dmath*}
Dividing by $T\to +\infty$ we obtain $\lambda_1\leq \lambda_2$. Obviously the same is true when exchanging the roles of $\lambda_1$ and $\lambda_2$. Therefore $\lambda_1 = \lambda_2$. 
}\fi

\if{

\begin{proof}[Proof of Corollary \ref{cor:ergodicity}]

First we perform the same change of unknown as in the proof of Theorem~\ref{thm:optimal growth}:  $\log v(t,x) = \log \langle\un,x\rangle + w(t,y)$. The latter satisfies the following Hamilton-Jacobi equation in the viscosity sense,
\begin{equation} \label{eq:HJB} 
\left\{\begin{array}{l}
\dfrac{\partial}{\partial t}(- w(t,y)) + H(D_y w(t,y),y) = 0\, , \quad t>0\, , \, y\in \S\, ,  \medskip \\
w(0,y) = w_0(y) := \log v_0(y) \, ,
\end{array}\right.
\end{equation}
By the dynamic programming principle, the value function reads also
\[ w(t,y) = \sup_{M\in L^\infty(0,t)} \left\{ \int_0^t \l(y_M(s),M(s))\, ds + w_0(y_M(t)) \right\}\, .  \]
Since $y_M(t)\in K_1$ for $t\geq 1$, we have $w_0(y_M(t))\in L^\infty(1,T)$, uniformly with respect to $T\geq 1$. Therefore we are reduced to proving that  
\[ \lim_{t\to +\infty} \dfrac{1 }{t} \sup_{M\in L^\infty(0,t)} \left\{ \int_0^{t} \l(y_M(s),M(s))\, ds \right\} = \lambda\]
uniformly with respect to $y\in \S.$
This statement is a classical consequence of the uniform convergence of $\eps u_\eps$~\cite{Arisawa1,Arisawa2,Bardi-Capuzzo}.
}\fi

\if{
 We repeat the argument below for the sake of completeness. 

Let $\delta>0$ be fixed, and $T = \frac \delta \eps$. We deduce from the dynamic programming principle that the optimal reward can be splitted as follows:
\begin{dmath}\label{eq:cor}
\eps u_\eps(y) = \sup_{M\in L^\infty(0,T)} \left\{ \eps\int_0^{T} \l(y_M(t),M(t))\, dt + \eps \int_0^{T} \left( e^{-\eps t} - 1 \right)\l(y_M(t),M(t))\, dt  + \eps e^{-\eps T}   u_\eps(y_M(T)) \right\}\, .
\end{dmath}
We have for the second contribution,
\begin{equation*}
\left| \eps \int_0^{T} \left( e^{-\eps t} - 1 \right)\l(y_M(t),M(t))\, dt \right| \leq \|\l\|_\infty \left( e^{-\delta } - 1 + \delta \right)\, .
\end{equation*}
Therefore, dividing \eqref{eq:cor} by $\delta$ we get as $\eps\to 0$ (or equivalently $T\to +\infty$),
\begin{dmath*}
 \liminf_{T\to +\infty} \dfrac{1 }{T} \sup_{M\in L^\infty(0,T)} \left\{ \int_0^{T} \l(y_M(t),M(t))\, dt \right\} +  \dfrac{\left( e^{-\delta} - 1\right)}\delta \lambda = O(\delta)\, , \\
 \limsup_{T\to +\infty} \dfrac{1 }{T} \sup_{M\in L^\infty(0,T)} \left\{ \int_0^{T} \l(y_M(t),M(t))\, dt \right\} +  \dfrac{\left( e^{-\delta} - 1\right)}\delta \lambda = O(\delta)\, ,
\end{dmath*}
Convergence occurs uniformly with respect to $y\in \S$ because of the uniform convergence $\epsilon u_\epsilon$ toxards $\lambda.$
Since $\delta>0$ can be chosen arbitrarily small, Corollary \ref{cor:ergodicity} is proven.
}\fi

\section{Qualitative properties of the optimal exponent $\lambda$.}\label{sec-qualitative}
\subsection{Optimality of stationnary controls in dimension $2$}

\begin{proposition}[Optimality and relaxed control]
The optimal growth rate $\lambda(\M)$ is greater or equal than any Perron eigenvalue $\lambda(m)$ for $m\in \M$.
\end{proposition}

\begin{proof}
An immediate proof of this statement is obtained by choosing a constant control $M\equiv m $ in \eqref{eq:fixed point}. We denote by $z_{m}\in \S$ the corresponding eigenvector. Since $z_m$ is a stationary point for the dynamics, we have 
\begin{align*}
\lambda(\M) t  + \mathring u(z_{m }) & \geq \int_0^t  \l(z_{m } ,m )\, ds + \mathring u(z_{m }) \\
& \geq    \lambda(m) t +  \mathring u(z_{m }) \, . \end{align*}
Therefore $\lambda(\M) \geq \lambda(m)$. A similar proof is obtained by noticing that $u_m(y) = \log \la \phi_m, y \ra$ is a supersolution  of \eqref{eq:ergodic HJ}. 
\end{proof}

Our next result shows that in dimension $2$, the optimal growth is achieved
by constant controls.
\begin{theorem} \label{th:dim n=2}
Assume that $n=2$. Then 
\[\lambda(\M) = \max_{m\in \M } \lambda(m)\, .\]
\end{theorem}

We skip the proof of this result, which exploits the Pontryagin maximum
principle, but rather give an heuristic argument.
The weak KAM statement, {\em i.e.} the existence of a pair $(\lambda, \mathring u)$ solution of the stationary Hamilton-Jacobi equation, generates an optimal vector field $b^*$. It is determined by the rule $b^*(y) = b(y, m^*)$ where $m^*\in \M$ realizes the maximum of the Hamiltonian $H$ in~\eqref{eq:ergodic HJ}.
Since Equation 
\eqref{eq:ergodic HJ}
 is stationary, the vector field $b^*$ is autonomous. However it is not defined everywhere on the simplex. For instance it cannot be defined on the points where $\mathring u$ is not differentiable nor on the points where the maximum value of $H$ is attained for several $m^*\in \M$. Anyway, up to this regularity issue, an autonomous vector field on the one-dimensional simplex is expected to exhibit fairly simple dynamics, {\em e.g.} convergence towards an equilibrium point. Simple arguments show that equilibria are in fact Perron eigenvectors. By optimality they have to be associated with the maximal possible eigenvalue for $m\in \M$.

A stronger result (where the unique optimal control is exhibited) can be found in~\cite{CoronGabrielShang} in a particular case coming from the modelling of the PMCA.

\if{
For any $m\in\M$ we denote by $z_m\in \S$ the projection of its Perron eigenvector.
By continuity of the eigenelements, the subset $Z = \{z_m\,  ,\, m\in \M\}$ is a segment of $\S$.
Consider the optimal control problem with the initial condition $y(0)=\min Z$ and let $M^*(t)\in\M$ be an optimal relaxed control.
We have from the Pontryagin maximum principle that the  associated trajectories $y(t)$ and $p(t)$ satisfy
\begin{align*} 
\dot y(t)  &= b(y(t), M^*(t)) \\
\dot p(t)  &= \la p(t), D_y b(y(t), M^*(t) )\ra  + D_y \l(y(t),M^*(t))\\
&\!\!\!\l (y(t),M^*(t)) + \la p(t),b(y\vc{(t)},M^*(t))\ra   =
\\
& \max_{m\in \M} \left( \l (y(t),m) + \la p(t),b(y(t),m)\ra  \right) 
 \end{align*}
Define the lower semi-continuous function
\begin{dmath*}\overline y(t)=\min_{m\in\M}\left\{ z_m ,\ m\in\argmax_{m'\in\M}(\l(y(t),m')+\la p(t),b(y(t),m')\ra)\right\}\end{dmath*}

By definition of $\overline y$ and from the choice of $y(0),$ we have $y(0)\leq\overline y(0).$
In the case $y(0)=\overline y(0),$ for any $m$ such that $y(0)=\overline y(0)=z_m$ we have
\[ \l (y(0),m) + \la p(0),b(y(0),m)\ra =  \l (z_m,m) = \lambda(m)\, . \] 
Recalling that the calibrated curve $y(t)$ travels with constant Hamiltonian 
\[ (\forall t\geq 0)\quad \l (y(t),m(t)) + \la p(t),b(y(t),m(t))\ra  = \lambda\,, \]
we conclude that $\lambda=\lambda(m).$

Assume $y(0)<\overline y(0),$ 
and define the first crossing time $T:=\sup\{t: (\forall s<t)\ y(s)<\overline y(s)\}\in(0,+\infty]$. If $T = +\infty$
then $y(t)$ is increasing and converges to a limit $y_\infty$ when $t\to+\infty$. Monotonicity of $y(t)$ is a consequence of the simple features of the   vector fields  $b(y,m)$ (for fixed $m\in \M$)  on the one-dimensional simplex $\S$: they all point to the corresponding eigenvector $z_m$. The latter is greater than $y(t)$ for all $t\geq 0$ by definition of $T = +\infty$. We have in addition $y_\infty\leq \max Z$ because $(\forall t\geq 0)$ $y(t)< \overline y(t) \leq \max Z$.
So there exists $m\in\M$ such that $y_\infty=z_m$. Assume by contradiction that $\liminf \overline y(t)> y_\infty = z_m$. Then by the same argument as above concerning the vector fields $b(y,m)$  (for fixed $m\in \M$), we deduce that the optimal trajectory $y(t)$ cannot be asymptotically stationary at $z_m$. Therefore $\overline y(t)$ accumulates at $z_m$ for a subsequence $(t_n)$. 
Passing to the limit $t_n\to+\infty$ in the Hamiltonian, we get as previously that $\lambda=\lambda(m)$.

It remains to consider the case where $0<T<+\infty$.
We have $y(T)\geq\overline y(T).$
If $y(T)=\overline y(T)$ we conclude as in the case  $y(0)=\overline y(0).$
On the contrary if $y(T)>\overline y(T)$ then $\overline y$ is discontinuous at $t=T$.
Denote by $m^-$ (resp. $m^+$) an element of $\M$ such that $z_{m^-}= \liminf_{t\to T^-}\overline y(t)$ (resp. $z_{m^+}=\overline y(T)$).
The whole segment $[m^-,m^+]$ is included in $\argmax_{m\in\M}(\l(y(T),m)+\la p(T),b(y(T),m)\ra)$. On the other hand  there exists $m\in[m^-,m^+]$ such that $y(T)=z_m$, simply because $y(T)\in[z_{m^+},z_{m^-}]$.
We deduce that $\lambda=\lambda(m)$.
}\fi

\

\subsection{Floquet perturbations of the maximal Perron eigenvalue.}
In this subsection, we give a few insights why we 
cannot hope 
for 
$\lambda(\M) = \max_{m\in  \M } \lambda(m)$ in dimension $n\geq 3$. We shall focus on the possible existence of limit cycles on the simplex which have a better reward than the maximal Perron eigenvalue. 

The arguments used to justify Theorem~\ref{th:dim n=2} 
cannot be transposed to a higher dimension.
\if{
First let us comment on the proof of Theorem \ref{th:dim n=2}. One intuition to rule out limit cycles in dimension $n = 2$ (one-dimensional simplex) is the following argument. From the relation $p(t) = D_y \mathring u(y(t))$ we notice that the optimal control at some given position $y\in \S_+$ is uniquely characterized by the Pontryagin maximum principle under the two following conditions: (i) $y$ belongs to an optimal trajectory for which it is not the starting point, and (ii) $y$ is not a "switching point" where the optimal control discontinuously jumps from one extremal of $\M$ to another. In fact condition (i) justifies the differentiability of $\mathring u$ at $y$~\cite{Fathi-book}.
We believe that conditions (i) and (ii) are generically true for almost all $y\in \S_+$. Therefore optimal trajectories "pick up" a unique control almost everywhere. This clearly rules out limit cycles on the one-dimensional simplex because such a cycle should go back and forth using different controls. The same argument cannot be transposed to higher dimension. }\fi
Another way to attack the problem is to test the optimal Perron eigenvalue against periodic perturbations. The question goes as follows: is it possible to find a larger Floquet eigenvalue in the neighbourhood of the maximal Perron eigenvalue? To address this issue we consider a simplified framework where $\M$ is a segment. We denote $\M = \{ G + \alpha F\, , \, \alpha \in [a,A]  \}$, and $\lambda(\alpha) = \lambda(G + \alpha F)$. We assume that there exists $\alpha^*\in (a,A)$ such that $\lambda(\alpha^*)$ is a local maximum of $\lambda(\alpha)$.

We assume for the sake of simplicity that the matrix $G + \alpha^* F$ is diagonalizable. 
We  denote by $(e_1^*,\cdots,e_n^*)$ and $(\phi_1^*,\cdots,\phi_n^*)$ the bases of right- and left- eigenvectors associated to the eigenvalues $\lambda_1^*>\lb_2^*\geq\cdots\geq \lb_n^*$ for the the best constant control $\alpha^*$, where $\lambda_1^* = \lambda(\alpha^*)$ is the Perron eigenvalue.
We recall the first order condition for $\lambda(\alpha^*)$ being a local maximum, 
\[ \phi_1^* F e_1^* = 0\, . \]

We consider small periodic perturbations of the best constant control: $\alpha(t) = \alpha^* + \epsilon \gamma(t)$, where $\gamma$ is a given $T$-periodic function. There exists a periodic eigenfunction $e_{\alpha^* + \epsilon \gamma}(t)$ associated to the Floquet eigenvalue $\lambda_F(\alpha^* + \epsilon \gamma)$ such that 
\begin{dmath*}\frac{\partial}{\partial t} e_{\alpha^* + \epsilon \gamma}(t) + \lb_F(\alpha^* + \epsilon\gamma ) e_{\alpha^* + \epsilon \gamma}(t) = (G+(\alpha^*+ \epsilon \gamma(t))F) e_{\alpha^* + \epsilon \gamma}(t)\, .\end{dmath*}
The following Proposition gives the second order condition for $\lambda(\alpha^*)$ being a local maximum relatively to periodic perturbations of the control. 
We denote by $\langle f\rangle_T$ the time average over one period,
\[\langle f\rangle_T=\f1T\int_0^T f(t)\,dt\, .\]

\begin{proposition}\label{prop:per:firstderiv}
The directional derivative of the Floquet eigenvalue vanishes at $\epsilon = 0$:
\begin{equation} \left.\dfrac{d \lb_F(\alpha^* + \epsilon \gamma)}{d\epsilon}\right|_{\epsilon = 0}  = 0\, . \label{eq:1st order}\end{equation}
Hence, $\alpha^*$ is also a critical point in the class of periodic controls. The second directional derivative of the Floquet eigenvalue writes at $\epsilon = 0$: 
\begin{equation}\label{eq:secondFloquet}
\left.\dfrac{d^2 \lb_F(\alpha^* + \epsilon \gamma)}{d\epsilon^2}\right|_{\epsilon = 0} = 2\sum_{i=2}^{n}  \langle\gamma_i^2\rangle_T\frac{(\phi_1^* F e_i^*)(\phi_i^* F e_1^*)}{ \lambda_1^*-\lambda_i^* }\, ,\end{equation}
where $\gamma_i(t)$ is the unique $T$-periodic solution of the relaxation ODE
\[
\dfrac{\dot{\gamma_i}(t)}{ \lambda_1^* - \lambda_i^* }  +  \gamma_i(t) =  \gamma(t) \, .
\]
\end{proposition}

The idea of computing directional derivatives has been used in a similar context in~\cite{M2} for optimizing the Perron eigenvalue in a continuous model for cell division. See also~\cite{Clairambault}
for a more general discussion on the comparison between Perron and Floquet eigenvalues. 
\if{
\vc{
\begin{proof}[Proof of Proposition~\ref{prop:per:firstderiv}]
First we derive a formula for the first derivative of the Perron eigenvalue:
By definition we have
\[(G+\al F)e_\alpha=\lambda_P(\alpha)e_\alpha\,.\]
Deriving with respect to $\al$ we get
\[
\frac {d\lambda_P}{d\alpha}(\alpha) e_\alpha+\lambda_P(\alpha)\frac {d e_\alpha}{d\alpha} =Fe_\alpha+(G+\al F)\frac {d e_\alpha}{d\alpha}\, .
\]
Testing against the left- eigenvector $\phi_\alpha$ we obtain
\[\frac {d\lambda_P}{d\alpha}(\alpha)    =  \phi_\alpha Fe_\alpha \, .\]
Second, we write the Floquet eigenvalue problem corresponding to the periodic control $\alpha = \alpha^* + \epsilon\gamma$:
\[\frac{\partial}{\partial t} e_\alpha(t) + \lb_F(\alpha^* + \epsilon\gamma ) e_\alpha(t) = (G+(\al^*+ \epsilon \gamma(t))F) e_\alpha(t)\, .\]
Deriving this ODE with respect to $\epsilon$, we get
\begin{dmath}\label{eq:firstderiv_periodic} 
\frac{\partial}{\partial t} \dfrac{\partial e_\alpha}{\partial\epsilon}(t) + \dfrac{d \lb_F(\alpha^* + \epsilon\gamma )}{d\epsilon} e_\alpha(t) + \lb_F(\alpha^* + \epsilon\gamma ) \dfrac{\partial e_\alpha}{\partial\epsilon}(t)  =  \gamma(t)F e_\alpha(t) + (G+(\al^*+ \epsilon \gamma(t))F)\dfrac{\partial e_\alpha}{\partial\epsilon}(t)  \, .
\end{dmath}
Testing this equation against $\phi_1^*$ and evaluating at $\epsilon = 0$, we obtain
\[ \frac{\partial}{\partial t} \left( \phi_1^* \left.\dfrac{\partial e_\alpha}{\partial\epsilon}\right|_{\epsilon = 0}(t) \right) + \left.\dfrac{d \lb_F(\alpha^* + \epsilon \gamma)}{d\epsilon}\right|_{\epsilon = 0} = \gamma(t) \phi_1^*F e_1^*\, .\]
After integration over one period, we get
\begin{dmath*} \left.\dfrac{d \lb_F(\alpha^* + \epsilon \gamma)}{d\epsilon}\right|_{\epsilon = 0} = \left(\frac1\theta\int_0^\theta \gamma(t)\,dt\right)\phi_1^*F e_1^*= \la \gamma \ra_\theta \dfrac{d\lambda_P}{d\alpha} (\al^*) = 0\, ,\end{dmath*}
which is the first order condition \eqref{eq:1st order}.

Next, we test~\eqref{eq:firstderiv_periodic} against another left-eigenvector $\phi_i^*$ and we evaluate at $\epsilon = 0$.
We obtain the following equation satisfied by $\gamma_i(t) = (\lambda_1^* - \lambda_i^*)\phi_i^* \dfrac{\partial e_\alpha}{\partial\epsilon}(t)(\phi_i^* Fe_1^*)^{-1}$:
\begin{equation} 
\frac1{\lambda_1^* - \lambda_i ^*} \frac \partial {\partial t}{\gamma_i}(t)  +  \gamma_i(t) =  \gamma(t)  \, .
\label{eq:gamma_i}
\end{equation}
We differentiate~\eqref{eq:firstderiv_periodic} with respect to $\epsilon$. This yields
\begin{dmath*}
\frac{\partial}{\partial t} \dfrac{\partial^2 e_\alpha}{\partial\epsilon^2}(t) + \dfrac{d^2 \lb_F(\alpha^* + \epsilon\gamma )}{d\epsilon^2} e_\alpha(t) + 2 \dfrac{d \lb_F(\alpha^* + \epsilon\gamma )}{d\epsilon} \dfrac{\partial  e_\alpha}{\partial\epsilon }(t)   +  \lb_F(\alpha^* + \epsilon\gamma ) \dfrac{\partial^2 e_\alpha}{\partial\epsilon^2}(t) 
\\ = 2 \gamma(t)F \dfrac{\partial  e_\alpha}{\partial\epsilon }(t) + (G+(\al^*+ \epsilon \gamma(t))F)\dfrac{\partial^2 e_\alpha}{\partial\epsilon^2}(t)  \, .
\end{dmath*}
Testing this equation against $\phi_1^*$ and evaluating at $\epsilon = 0$, we find
\begin{dmath}
\frac{\partial}{\partial t} \left( \phi_1^* \left.\dfrac{\partial^2 e_\alpha}{\partial\epsilon^2}\right|_{\epsilon = 0}(t) \right) +  \left.\dfrac{d^2 \lb_F(\alpha^* + \epsilon \gamma)}{d\epsilon^2}\right|_{\epsilon = 0}    = 2\gamma(t) \phi_1^* F   \left.\dfrac{\partial e_\alpha}{\partial\epsilon}\right|_{\epsilon = 0}(t)  \, . 
\label{eq:extremal Floquet}
\end{dmath}
We decompose the unknown $\frac{\partial e_\alpha}{\partial\epsilon}(t)$ along the basis $(e_1^*,e_2^*,e_3^*)$
\[  \frac{\partial e_\alpha}{\partial\epsilon} (t) = \sum_{i = 1}^3 \gamma_i(t)\frac{(\phi_i^* Fe_1^*)}{\lambda_1^* - \lambda_i^*}  e_i^*  \, .\]
In particular, we have
\[ \phi_1^* F    \dfrac{\partial e_\alpha}{\partial\epsilon} (t) = \sum_{i=2}^{3}\gamma_i(t)\frac{(\phi_i^* F e_1^*)}{\lambda_1^* - \lambda_i^*}(\phi_1^* F e_i^*)\, ,   \]
since $\phi_1^* Fe_1^* = 0$ by optimality.
To conclude, we integrate \eqref{eq:extremal Floquet} over one period,
\begin{equation*}
\left.\dfrac{d^2 \lb_F(\alpha^* + \epsilon \gamma)}{d\epsilon^2}\right|_{\epsilon = 0}  =
 2 \sum_{i=2}^{3}\langle\gamma\gamma_i\rangle_\theta \dfrac{(\phi_i^* F e_1^*)( \phi_1^* F e_i^*)}{\lambda_1^* - \lambda_i^*} \,.
\end{equation*}
We conclude thanks to the following identity derived from \eqref{eq:gamma_i}: $\la \gamma_i ^2\ra_\theta = \langle\gamma\gamma_i\rangle_\theta$.
\end{proof}
}
}\fi

Taking $\gamma\equiv1$ in Equation~\eqref{eq:secondFloquet}, we get the second derivative of the Perron eigenvalue at $\alpha^*$,
\begin{equation}\label{eq:secondPerron}\dfrac{d^2\lambda}{d\alpha^2}(\alpha^*)=2\sum_{i=2}^{n}\dfrac{(\phi_1^* F e_i^*)(\phi_i^* F e_1^*)}{\lambda_1^*-\lambda_i^*} \, ,\end{equation}
which is nonpositive since $\alpha^*$ is a maximum point. 
Therefore we are led to the following question: is it possible to construct counter-examples such that the sum 
\eqref{eq:secondFloquet} is positive for some periodic control $\gamma$, whereas the sum \eqref{eq:secondPerron} is nonpositive? This is clearly not possible in dimension $n = 2$ because the sum in \eqref{eq:secondFloquet} is reduced to a single nonpositive term by \eqref{eq:secondFloquet}. 
For $n\geq3,$ considering periodic perturbations $\gamma(t)=\cos(\omega t),$ we get the formula
\[\left.\dfrac{d^2 \lb_F(\alpha^* + \epsilon \gamma)}{d\epsilon^2}\right|_{\epsilon = 0}  \!\!\!\!=
 \sum_{i=2}^{n} \dfrac{\lambda_1^* - \lambda_i^*}{\omega^2+(\lambda_1^* - \lambda_i^*)^2}(\phi_i^* F e_1^*)( \phi_1^* F e_i^*) \,.\]
An asymptotic expansion when $\omega\to+\infty$ indicates that if the condition
\begin{equation}\label{eq:secondFloquetomega}
\sum_{i=2}^{n} (\lambda_1^* - \lambda_i^*)(\phi_i^* F e_1^*)( \phi_1^* F e_i^*)>0
\end{equation}
is satisfied, then~\eqref{eq:secondFloquet} is positive for some $\omega$ large enough.




\subsection{Legendre type condition for local optimality on short times}

Within the framework described in the previous section, we introduce the endpoint mapping
\[F_T:\left\{\begin{array}{rcl}
L^\infty(0,T)&\to& K,\\
\alpha(\cdot)&\mapsto& x(T),
\end{array}\right.\]
which maps a control $\alpha\in L^\infty(0,T)$ to the terminal value $x(T)$ of the corresponding trajectory, i.e. the solution of the ODE
\[\dot x(s)=(G+\alpha(s) F)x(s)\,.\]
We analyse in the following the behaviour of this mapping in the neighbourhood of the best constant control $\alpha^*$ and its associated trajectory $x(t)=e^{\lambda_1^*t}e_1^*,$. Moreover  we make the link with the computations on the Floquet eigenvalue in the previous section.
Consider a variation $\alpha(\cdot)=\alpha^*+\epsilon\,\gamma(\cdot)\in L^\infty(0,T)$ (not necessarily periodic) and define the quadratic form
\[Q(\gamma):=\phi_1^*(D^2_{\alpha^*}F_T)(\gamma,\gamma).\]
\if{The corresponding trajectory $y$ is the solution to the equation
\[\dot y=(G+\alpha^*F)y+\epsilon\,\gamma Fy,\qquad y(0)=e_1^*,\]
and we deduce the asymptotic expansion of the endpoint mapping
\begin{align}
F_T(\alpha)&=x(T)+\epsilon\int_0^T\gamma(t)e^{(G+\alpha^*F)(T-t)}Fx(t)\,dt
\nonumber%
\\
&\!\!\!\!\!\!\!\!\!\!\!\!\!\!\!\!\!\!\!+\epsilon^2\!\!\!\int_0^T\!\!\!\!\int_0^t\!\!\!\!\gamma(t)\gamma(s)e^{(G+\alpha^*F)(T-t)}Fe^{(G+\alpha^*F)(t-s)}Fx(s)\,dsdt\nonumber\\ 
& +O(\epsilon^3).
\label{eq:expansion}
\end{align}
Then the first variation of $F_T$ is given by
\[(D_{\alpha^*}F_T)\gamma=\int_0^T\gamma(t)e^{(G+\alpha^*F)(T-t)}Fx(t)\,dt.\]
The control $\alpha^*$ is said to be a critical point of $F_T$ if there exists a Lagrange multiplier $\psi\in\R^n$ such that
\[\psi(D_{\alpha^*}F_T)\gamma=0,\quad\forall\gamma,\]
i.e.,
\begin{equation}\label{first_order_condition}\psi e^{(G+\alpha^*F)(T-t)}Fx(t)=0,\quad\forall t\in[0,T].\end{equation}
Equivalently we can say that $\alpha^*$ is a critical point if and only if there exists a solution $p(t)$ to the adjoint equation
\[\dot p=-p(G+\alpha^* F)\]
such that
\begin{equation}\label{first_order_condition_PMP}\forall t\in[0,T],\quad p(t)Fx(t)=0.\end{equation}
The dual Perron eigenvector $\phi_1^*$ associated to $\alpha^*$ satisfies condition~\eqref{first_order_condition}
\[\phi_1^* e^{(G+\alpha^*F)(T-t)}Fx(t)=e^{\lambda_1^*t}\phi_1^*Fe_1^*=0.\]
In terms of our optimal control problem, that means that  the control $\alpha^*$ is critical when the objective we want to maximize is $\phi_1^*x(T).$
Condition~\eqref{first_order_condition_PMP} is satisfied with $p(t)=e^{\lambda_1^*(T-t)}\phi_1^*.$

We go further to second order condition.
The asymptotic expansion~\eqref{eq:expansion} yields the expression for the second differential
\begin{dmath*}(D^2_{\alpha^*}F_T)(\gamma,\gamma)=
2\!\!\int_0^T\!\!\!\!\int_0^t\!\!\!\!\!\gamma(t)\gamma(s)e^{(G+\alpha^*F)\!(T-t)}\!F\!e^{(G+\alpha^*F)\!(t-s)}\!F\!x(s)ds\!dt.\end{dmath*}
Take the Lagrange multiplier $\phi_1^*$ and define the corresponding quadratic form
\begin{align}
Q(\gamma)&:=\phi_1^*(D^2_{\alpha^*}F_T)(\gamma,\gamma)\nonumber\\
&\!\!\!\!\!\!\!\!\!\!\!\!\!\!\!=2\!\!\int_0^T\!\!\!\!\!\int_0^t\!\!\!\!\gamma(t)\gamma(s)\,\!\phi_1^*\!e^{(G+\alpha^*F)\!(T-t)}\!F\!e^{(G+\alpha^*F)\!(t-s)}\!F\!x(s)dsdt\nonumber\\
&\!\!\!\!\!\!\!\!\!\!\!\!\!\!\!=2e^{\lambda_1^*T}\!\!\!\int_0^T\!\!\!\!\!\int_0^t\!\!\!\!\gamma(t)\gamma(s)\,\phi_1^*F\!e^{(G+\alpha^*F-\lambda_1^*I)\!(t-s)}F\!e_1^*dsdt.\label{eq:quadratic_form}
\end{align}}\fi
A straightforward computation gives the expression
\[Q(\gamma)=2e^{\lambda_1^*T}\!\!\!\int_0^T\!\!\!\!\!\int_0^t\!\!\!\!\gamma(t)\gamma(s)\,\phi_1^*F\!e^{(G+\alpha^*F-\lambda_1^*I)\!(t-s)}F\!e_1^*dsdt\]
Looking at the leading terms in $Q$ when $T$ is small, we get a sufficient condition for the control $\alpha^*$ to be locally optimal for small times on the hyperplane
\[L^\infty_0:=\Bigl\{\gamma\in L^\infty(0,T),\ \int_0^T\gamma(t)\,dt=0\Bigr\}.\]

\begin{proposition}\label{prop:Legendre}
If the condition
\begin{equation}\label{legendre_condition}\phi_1^*F(G+\alpha^*F-\lambda_1^*I)Fe_1^*>0.\end{equation}
is satisfied, then the quadratic form $Q$ restricted to $L^\infty_0$ is negative definite for short times with respect to the negative Sobolev space $H^{-1}(0,T)$, that is to say
\[\exists\delta>0,\ \exists\epsilon>0,\ \forall T\in(0,\epsilon),\quad Q|_{L^\infty_0}(\gamma)\leq-\delta\|\gamma\|_{H^{-1}}^2.\]
\end{proposition}
We skip the proof of this result.
\if{
\begin{proof}
The leading term (when $T\to0$) in~\eqref{eq:quadratic_form} vanishes on $L^\infty_0$
\[2\int_0^T\int_0^t\beta(t)\beta(s)\,\phi^*F^2e_1^*\,ds\,dt
=\biggl(\int_0^T\beta(t)\,dt\biggr)^2\phi^*F^2e_1^*.\]
The following dominating term is
\[2e^{\lambda^*T}\int_0^T\int_0^t\beta(t)\beta(s)(t-s)\,\phi^*F(G+\alpha^*F-\lambda^*I)Fe_1^*\,ds\,dt\]
and two integrations by parts give
\begin{dmath*}\int_0^T\beta(t)\int_0^t\beta(s)(t-s)\,ds\,dt=\int_0^T\beta(t)\,dt\int_0^T B(t)\,dt-\int_0^TB^2(t)\,dt\end{dmath*}
where $B(t)=\int_0^t\beta(s)\,ds.$
We get the result since the first term on the r.h.s. vanishes on $L^\infty_0$ and the second term is nothing but the squared $H^{-1}$ norm of $\beta.$ 
\end{proof}}\fi


Condition~\eqref{legendre_condition} is the so-called {\it generalized Legendre condition} of our problem.
The generalized Legendre condition appears in the study of optimality for totally singular extremals, i.e. when the second derivative of the Hamiltonian is identically zero along the trajectory.
A typical example is provided by the single-input affine control systems, namely, $\dot x(t)=f_0(x(t))+\alpha(t)f_1(x(t)),$ where $\alpha(t)\in\R,$ and $f_0,f_1$ are smooth vector fields.
In this case the generalized Legendre condition writes
\[\langle p(\cdot)[f_1,[f_0,f_1]]x(\cdot)\rangle>0\]
where $[\cdot,\cdot]$ is the Lie bracket of vector fields.
Our linear control system belongs to this class of problems, and straightforward computations show that for $x(t)=e^{\lambda_1^*t}e_1^*$ and $p(t)=e^{-\lambda_1^*t}\phi_1^*,$ we have $p(\cdot)[F,[G,F]]x(\cdot)=\phi_1^*F(G+\alpha^*F-\lambda_1^*I)Fe_1^*.$
The generalized Legendre condition ensures that the quadratic form $Q|_{\text{Ker}D_{\alpha^*}F_T}$ is definite negative.
Then it allows to deduce that the trajectory $x(\cdot)$ is locally optimal for short final times $T$ in the $C^0$ topology (we refer to~\cite{AgrachevSachkov}
 and~\cite{BonnardCaillauTrelat} for details).

Here we proved the negativity of $Q$ on $L^\infty_0$ instead of $\text{Ker}D_{\alpha^*}F_T$ under the condition \eqref{prop:Legendre}. Our aim is to make clearer the link with the computation of the second derivative of the Floquet eigenvalue.
It follows from the previous section that the second derivative of $\lambda_F$ is positive for periodic controls $\cos(\omega t)$ when $\omega$ is large if condition~\eqref{eq:secondFloquetomega} is satisfied.
Considering $T$ small and $\omega=k\frac{2\pi}{T}$ with $k\geq1,$ we have that $\cos(\omega t)\in L^\infty_0$ and $\omega\to+\infty$ when $k\to+\infty.$
The following proposition points out the consistency between conditions~\eqref{legendre_condition} and~\eqref{eq:secondFloquetomega}.
\begin{proposition}\label{prop:equivalence_Legendre_Floquet}
We have
\[\phi_1^*F(G+\alpha^*F-\lambda_1^*I)Fe_1^*=-\sum_{i=2}^n(\lambda_1^*-\lambda_i^*)(\phi_1^*Fe_i^*)(\phi_i^*Fe_1^*).\]
\end{proposition}
This relation is instructive since it emphazises the relation between a condition for optimality for small times and a condition for optimality with high frequences.
\if{
\begin{proof}
We decompose $Fe_1^*$ on the basis $(e_i^*)_{1\leq i\leq n}$
\[Fe_1^*=\sum_{i=1}^n(\phi_i^*Fe_1^*)\,e_i^*.\]
Then we get
\[\lambda^*\phi_1^*F^2e_1^*=\lambda^*\sum_{i=1}^n(\phi_i^*Fe_1^*)(\phi_1^*Fe_i^*)\]
and
\[\phi_1^*F(G+\alpha^*F)Fe_1^*=\sum_{i=1}^n\lambda_i^*(\phi_i^*Fe_1^*)(\phi_1^*Fe_i^*)\]
\end{proof}}\fi

\if{\subsection*{Floquet perturbations of the maximal Perron eigenvalue.}
In this subsection, we give a few insights why we 
cannot hope 
for 
$\lambda = \max_{m\in  \co(\M) } \lambda(m)$ in dimension $n\geq 3$. We shall focus on the possible existence of limit cycles on the simplex which have a better reward than the maximal Perron eigenvalue. 

The arguments used to justify of Theorem~\ref{th:dim n=2} 
cannot be transposed to a higher dimension.
\if{
First let us comment on the proof of Theorem \ref{th:dim n=2}. One intuition to rule out limit cycles in dimension $n = 2$ (one-dimensional simplex) is the following argument. From the relation $p(t) = D_y \mathring u(y(t))$ we notice that the optimal control at some given position $y\in \S_+$ is uniquely characterized by the Pontryagin maximum principle under the two following conditions: (i) $y$ belongs to an optimal trajectory for which it is not the starting point, and (ii) $y$ is not a "switching point" where the optimal control discontinuously jumps from one extremal of $\M$ to another. In fact condition (i) justifies the differentiability of $\mathring u$ at $y$~\cite{Fathi-book}.
We believe that conditions (i) and (ii) are generically true for almost all $y\in \S_+$. Therefore optimal trajectories "pick up" a unique control almost everywhere. This clearly rules out limit cycles on the one-dimensional simplex because such a cycle should go back and forth using different controls. The same argument cannot be transposed to higher dimension. }\fi
Another way to attack the problem is to test the optimal Perron eigenvalue against periodic perturbations. The question goes as follows: is it possible to find a larger Floquet eigenvalue in the neighbourhood of the maximal Perron eigenvalue? To address this issue we consider a simplified framework: we assume that $n=3$ and that $\M$ is a segment. We denote $\M = \{ G + \alpha F\, , \, \alpha \in [a,A]  \}$, and $\lambda(\alpha) = \lambda(G + \alpha F)$. We assume that there exists $\alpha^*\in (a,A)$ such that $\lambda(\alpha^*)$ is a local maximum of $\lambda(\alpha)$.

We assume for the sake of simplicity that the matrix $G + \alpha^* F$ is real diagonalizable. 
We  denote by $(e_1^*,e_2^*,e_3^*)$ and $(\phi_1^*,\phi_2^*,\phi_3^*)$ the bases of right- and left- eigenvectors associated to the eigenvalues $\lambda_1^*>\lb_2^*\geq \lb_3^*$ for the the best constant control $\alpha^*$, where $\lambda_1^* = \lambda(\alpha^*)$ is the Perron eigenvalue.
Recall the first order condition for $\lambda(\alpha^*)$ being a local maximum, 
\[ \phi_1^* F e_1^* = 0\, . \]

We consider small periodic perturbations of the best constant control: $\alpha(t) = \alpha^* + \epsilon \gamma(t)$, where $\gamma$ is a given $T$-periodic function. There exists a periodic eigenfunction $e_{\alpha^* + \epsilon \gamma}(t)$ associated to the Floquet eigenvalue $\lambda_F(\alpha^* + \epsilon \gamma)$ such that 
\begin{dmath*}\frac{\partial}{\partial t} e_{\alpha^* + \epsilon \gamma}(t) + \lb_F(\alpha^* + \epsilon\gamma ) e_{\alpha^* + \epsilon \gamma}(t) = (G+(\alpha^*+ \epsilon \gamma(t))F) e_{\alpha^* + \epsilon \gamma}(t)\, .\end{dmath*}
The following Proposition gives the second order condition for $\lambda(\alpha^*)$ being a local maximum relatively to periodic perturbations of the control. 
We denote by $\langle f\rangle_T$ the time average over one period,
\[\langle f\rangle_T=\f1T\int_0^T f(t)\,dt\, .\]

\begin{proposition}\label{prop:per:firstderiv}
The directional derivative of the Floquet eigenvalue vanishes at $\epsilon = 0$:
\begin{equation} \left.\dfrac{d \lb_F(\alpha^* + \epsilon \gamma)}{d\epsilon}\right|_{\epsilon = 0}  = 0\, . \label{eq:1st order}\end{equation}
Hence, $\alpha^*$ is also a critical point in the class of periodic controls. The second directional derivative of the Floquet eigenvalue writes at $\epsilon = 0$: 
\begin{equation}\label{eq:secondFloquet}
\left.\dfrac{d^2 \lb_F(\alpha^* + \epsilon \gamma)}{d\epsilon^2}\right|_{\epsilon = 0} = 2\sum_{i=2}^{3}  \langle\gamma_i^2\rangle_T\frac{(\phi_1^* F e_i^*)(\phi_i^* F e_1^*)}{ \lambda_1^*-\lambda_i^* }\, ,\end{equation}
where $\gamma_i(t)$ is the unique $T$-periodic solution of the relaxation ODE
\[
\dfrac{\dot{\gamma_i}(t)}{ \lambda_1^* - \lambda_i^* }  +  \gamma_i(t) =  \gamma(t) \, .
\]
\end{proposition}

The idea of computing directional derivatives has been used in a similar context in~\cite{M2} for optimizing the Perron eigenvalue in a continuous model for cell division. See also~\cite{Clairambault}
for a more general discussion on the comparison between Perron and Floquet eigenvalues. 
\if{
\vc{
\begin{proof}[Proof of Proposition~\ref{prop:per:firstderiv}]
First we derive a formula for the first derivative of the Perron eigenvalue:
By definition we have
\[(G+\al F)e_\alpha=\lambda_P(\alpha)e_\alpha\,.\]
Deriving with respect to $\al$ we get
\[
\frac {d\lambda_P}{d\alpha}(\alpha) e_\alpha+\lambda_P(\alpha)\frac {d e_\alpha}{d\alpha} =Fe_\alpha+(G+\al F)\frac {d e_\alpha}{d\alpha}\, .
\]
Testing against the left- eigenvector $\phi_\alpha$ we obtain
\[\frac {d\lambda_P}{d\alpha}(\alpha)    =  \phi_\alpha Fe_\alpha \, .\]
Second, we write the Floquet eigenvalue problem corresponding to the periodic control $\alpha = \alpha^* + \epsilon\gamma$:
\[\frac{\partial}{\partial t} e_\alpha(t) + \lb_F(\alpha^* + \epsilon\gamma ) e_\alpha(t) = (G+(\al^*+ \epsilon \gamma(t))F) e_\alpha(t)\, .\]
Deriving this ODE with respect to $\epsilon$, we get
\begin{dmath}\label{eq:firstderiv_periodic} 
\frac{\partial}{\partial t} \dfrac{\partial e_\alpha}{\partial\epsilon}(t) + \dfrac{d \lb_F(\alpha^* + \epsilon\gamma )}{d\epsilon} e_\alpha(t) + \lb_F(\alpha^* + \epsilon\gamma ) \dfrac{\partial e_\alpha}{\partial\epsilon}(t)  =  \gamma(t)F e_\alpha(t) + (G+(\al^*+ \epsilon \gamma(t))F)\dfrac{\partial e_\alpha}{\partial\epsilon}(t)  \, .
\end{dmath}
Testing this equation against $\phi_1^*$ and evaluating at $\epsilon = 0$, we obtain
\[ \frac{\partial}{\partial t} \left( \phi_1^* \left.\dfrac{\partial e_\alpha}{\partial\epsilon}\right|_{\epsilon = 0}(t) \right) + \left.\dfrac{d \lb_F(\alpha^* + \epsilon \gamma)}{d\epsilon}\right|_{\epsilon = 0} = \gamma(t) \phi_1^*F e_1^*\, .\]
After integration over one period, we get
\begin{dmath*} \left.\dfrac{d \lb_F(\alpha^* + \epsilon \gamma)}{d\epsilon}\right|_{\epsilon = 0} = \left(\frac1\theta\int_0^\theta \gamma(t)\,dt\right)\phi_1^*F e_1^*= \la \gamma \ra_\theta \dfrac{d\lambda_P}{d\alpha} (\al^*) = 0\, ,\end{dmath*}
which is the first order condition \eqref{eq:1st order}.

Next, we test~\eqref{eq:firstderiv_periodic} against another left-eigenvector $\phi_i^*$ and we evaluate at $\epsilon = 0$.
We obtain the following equation satisfied by $\gamma_i(t) = (\lambda_1^* - \lambda_i^*)\phi_i^* \dfrac{\partial e_\alpha}{\partial\epsilon}(t)(\phi_i^* Fe_1^*)^{-1}$:
\begin{equation} 
\frac1{\lambda_1^* - \lambda_i ^*} \frac \partial {\partial t}{\gamma_i}(t)  +  \gamma_i(t) =  \gamma(t)  \, .
\label{eq:gamma_i}
\end{equation}
We differentiate~\eqref{eq:firstderiv_periodic} with respect to $\epsilon$. This yields
\begin{dmath*}
\frac{\partial}{\partial t} \dfrac{\partial^2 e_\alpha}{\partial\epsilon^2}(t) + \dfrac{d^2 \lb_F(\alpha^* + \epsilon\gamma )}{d\epsilon^2} e_\alpha(t) + 2 \dfrac{d \lb_F(\alpha^* + \epsilon\gamma )}{d\epsilon} \dfrac{\partial  e_\alpha}{\partial\epsilon }(t)   +  \lb_F(\alpha^* + \epsilon\gamma ) \dfrac{\partial^2 e_\alpha}{\partial\epsilon^2}(t) 
\\ = 2 \gamma(t)F \dfrac{\partial  e_\alpha}{\partial\epsilon }(t) + (G+(\al^*+ \epsilon \gamma(t))F)\dfrac{\partial^2 e_\alpha}{\partial\epsilon^2}(t)  \, .
\end{dmath*}
Testing this equation against $\phi_1^*$ and evaluating at $\epsilon = 0$, we find
\begin{dmath}
\frac{\partial}{\partial t} \left( \phi_1^* \left.\dfrac{\partial^2 e_\alpha}{\partial\epsilon^2}\right|_{\epsilon = 0}(t) \right) +  \left.\dfrac{d^2 \lb_F(\alpha^* + \epsilon \gamma)}{d\epsilon^2}\right|_{\epsilon = 0}    = 2\gamma(t) \phi_1^* F   \left.\dfrac{\partial e_\alpha}{\partial\epsilon}\right|_{\epsilon = 0}(t)  \, . 
\label{eq:extremal Floquet}
\end{dmath}
We decompose the unknown $\frac{\partial e_\alpha}{\partial\epsilon}(t)$ along the basis $(e_1^*,e_2^*,e_3^*)$
\[  \frac{\partial e_\alpha}{\partial\epsilon} (t) = \sum_{i = 1}^3 \gamma_i(t)\frac{(\phi_i^* Fe_1^*)}{\lambda_1^* - \lambda_i^*}  e_i^*  \, .\]
In particular, we have
\[ \phi_1^* F    \dfrac{\partial e_\alpha}{\partial\epsilon} (t) = \sum_{i=2}^{3}\gamma_i(t)\frac{(\phi_i^* F e_1^*)}{\lambda_1^* - \lambda_i^*}(\phi_1^* F e_i^*)\, ,   \]
since $\phi_1^* Fe_1^* = 0$ by optimality.
To conclude, we integrate \eqref{eq:extremal Floquet} over one period,
\begin{equation*}
\left.\dfrac{d^2 \lb_F(\alpha^* + \epsilon \gamma)}{d\epsilon^2}\right|_{\epsilon = 0}  =
 2 \sum_{i=2}^{3}\langle\gamma\gamma_i\rangle_\theta \dfrac{(\phi_i^* F e_1^*)( \phi_1^* F e_i^*)}{\lambda_1^* - \lambda_i^*} \,.
\end{equation*}
We conclude thanks to the following identity derived from \eqref{eq:gamma_i}: $\la \gamma_i ^2\ra_\theta = \langle\gamma\gamma_i\rangle_\theta$.
\end{proof}
}
}\fi

Taking $\gamma\equiv1$ in Equation~\eqref{eq:secondFloquet}, we get the second derivative of the Perron eigenvalue at $\alpha^*$,
\begin{equation}\label{eq:secondPerron}\dfrac{d^2\lambda}{d\alpha^2}(\alpha^*)=2\sum_{i=2}^{3}\dfrac{(\phi_1^* F e_i^*)(\phi_i^* F e_1^*)}{\lambda_1^*-\lambda_i^*} \, ,\end{equation}
which is nonpositive since $\alpha^*$ is a maximum point. 
Therefore we are led to the following question: is it possible to construct counter-examples such that the sum 
\eqref{eq:secondFloquet} is positive for some periodic control $\gamma$, whereas the sum \eqref{eq:secondPerron} is nonpositive? This is clearly not possible in dimension $n = 2$ because the sum in \eqref{eq:secondFloquet} is reduced to a single nonpositive term by \eqref{eq:secondFloquet}. 




\subsection*{Legendre type condition for local optimality on short times}

Consider the control system
\begin{align}\label{eq:control_syst_alpha}
&\dot x(t)=(G+\alpha(t)F) x(t),\\
&x(0)=e_1^*,\nonumber\\
&\alpha(\cdot)\in L^\infty([0,T],[a,A]),\nonumber
\end{align}
where $e_1^*$ is the Perron eigenvalue corresponding to the parameter $\alpha^*\in(a,A)$ which maximizes the Perron eigenvalue on the set $\{G+\alpha F,\ \alpha\in[a,A]\}.$
A particular admissible control is given by $\alpha(\cdot)\equiv\alpha^*,$ and the associated trajectory is given by $x(t)=e^{\lambda^*t}\,e_1^*.$
We study the endpoint mapping
\[F_T:\left\{\begin{array}{rcl}
L^\infty([0,T],[a,A])&\to& K,\\
\alpha(\cdot)&\mapsto& x(T),
\end{array}\right.\]
in the neighborhood of $\alpha^*.$
Consider a variation $\alpha(\cdot)=\alpha^*+\epsilon\,\beta(\cdot)\in L^\infty([0,T],[a,A]).$
The corresponding trajectory $y$ is given by
\[\dot y=(G+\alpha^*F)y+\epsilon\,\beta Fy,\qquad y(0)=e_1^*,\]
and we deduce the asymptotic expansion of the endpoint mapping
\begin{align}
F_T(\alpha)&=x(T)+\epsilon\int_0^T\beta(t)e^{(G+\alpha^*F)(T-t)}Fx(t)\,dt
\nonumber%
\\
&\!\!\!\!\!\!\!\!\!\!\!\!\!\!\!\!\!\!\!+\epsilon^2\!\!\!\int_0^T\!\!\!\!\int_0^t\!\!\!\!\beta(t)\beta(s)e^{(G+\alpha^*F)(T-t)}Fe^{(G+\alpha^*F)(t-s)}Fx(s)\,ds\,dt\nonumber\\ 
& +O(\epsilon^3).
\label{eq:expansion}
\end{align}
Then the first variation of $F_T$ is given by
\[(D_{\alpha^*}F_T)\beta=\int_0^T\beta(t)e^{(G+\alpha^*F)(T-t)}Fx(t)\,dt.\]
The control $\alpha^*$ is said to be a critical point of $F_T$ if there exists a Lagrange multiplier $\psi\in\R^n$ such that
\[\psi(D_{\alpha^*}F_T)\beta=0,\quad\forall\beta,\]
i.e.,
\begin{equation}\label{first_order_condition}\psi e^{(G+\alpha^*F)(T-t)}Fx(t)=0,\quad\forall t\in[0,T].\end{equation}
Equivalently we can say that $\alpha^*$ is a critical point if and only if there exists a solution $p(t)$ to the adjoint equation
\[\dot p=-p(G+\alpha^* F)\]
such that
\begin{equation}\label{first_order_condition_PMP}\forall t\in[0,T],\quad p(t)Fx(t)=0.\end{equation}
The dual Perron eigenvector $\phi^*$ associated to $\alpha^*$ satisfies condition~\eqref{first_order_condition}
\[\phi^* e^{(G+\alpha^*F)(T-t)}Fx(t)=e^{\lambda^* T}\phi^*Fe_1^*=0.\]
In terms of our optimal control problem, that means that  the control $\alpha^*$ is critical for the control system~\eqref{eq:control_syst_alpha} when the objective we want to maximize is $\phi^*x(T).$
Condition~\eqref{first_order_condition_PMP}, which is equivalent to the Pontryagin maximum principle, is satisfied with $p(t)=e^{\lambda^*(T-t)}\phi^*.$

We go further to second order condition.
The asymptotic expansion~\eqref{eq:expansion} yields the expression for the second differential
\begin{dmath*}(D^2_{\alpha^*}F_T)(\beta,\beta)=
2\!\!\int_0^T\!\!\!\!\int_0^t\!\!\!\!\!\beta(t)\beta(s)e^{(G+\alpha^*F)(T-t)}Fe^{(G+\alpha^*F)(t-s)}Fx(s)\,ds\,dt.\end{dmath*}
Take the Lagrange multiplier $\phi^*$ and define the corresponding quadratic form
\begin{align}
Q(\beta)&:=\phi^*(D^2_{\alpha^*}F_T)(\beta,\beta)\nonumber\\
&\!\!\!\!\!\!\!\!\!\!\!\!\!\!\!=2\!\!\int_0^T\!\!\!\!\!\int_0^t\!\!\!\!\beta(t)\beta(s)\,\phi^*e^{(G+\alpha^*F)(T-t)}Fe^{(G+\alpha^*F)(t-s)}Fx(s)dsdt\nonumber\\
&\!\!\!\!\!\!\!\!\!\!\!\!\!\!\!=2e^{\lambda^*T}\!\!\!\int_0^T\!\!\!\!\!\int_0^t\!\!\!\!\beta(t)\beta(s)\,\phi^*Fe^{(G+\alpha^*F-\lambda^*I)(t-s)}Fe_1^*dsdt.\label{eq:quadratic_form}
\end{align}
Computing the leading terms in $Q$ when $T$ is small, we get a sufficient condition for the control $\alpha^*$ to be locally optimal for small times on the hyperplane
\[L^\infty_0:=\Bigl\{u\in L^\infty([0,T],[a,A]),\ \int_0^Tu(t)\,dt=0\Bigr\}.\]

\begin{proposition}\label{prop:Legendre}
If the condition
\begin{equation}\label{legendre_condition}\phi^*F(G+\alpha^*F-\lambda^*I)Fe_1^*>0.\end{equation}
is satisfied, then the quadratic form $Q{|_{L^\infty_0}}$ is $H^{-1}$ negative definite for short terms, i.e.
\[\exists\gamma>0,\ \exists\epsilon>0,\ \forall T\in(0,\epsilon),\quad Q|_{L^\infty_0}(\beta,\beta)\leq-\gamma\|\beta\|_{H^{-1}}^2.\]
\end{proposition}
We skip the proof of this result.
\if{
\begin{proof}
The leading term (when $T\to0$) in~\eqref{eq:quadratic_form} vanishes on $L^\infty_0$
\[2\int_0^T\int_0^t\beta(t)\beta(s)\,\phi^*F^2e_1^*\,ds\,dt
=\biggl(\int_0^T\beta(t)\,dt\biggr)^2\phi^*F^2e_1^*.\]
The following dominating term is
\[2e^{\lambda^*T}\int_0^T\int_0^t\beta(t)\beta(s)(t-s)\,\phi^*F(G+\alpha^*F-\lambda^*I)Fe_1^*\,ds\,dt\]
and two integrations by parts give
\begin{dmath*}\int_0^T\beta(t)\int_0^t\beta(s)(t-s)\,ds\,dt=\int_0^T\beta(t)\,dt\int_0^T B(t)\,dt-\int_0^TB^2(t)\,dt\end{dmath*}
where $B(t)=\int_0^t\beta(s)\,ds.$
We get the result since the first term on the r.h.s. vanishes on $L^\infty_0$ and the second term is nothing but the squared $H^{-1}$ norm of $\beta.$ 
\end{proof}}\fi


Condition~\eqref{legendre_condition} is the so-called {\it generalized Legendre condition} of our problem.
The generalized Legendre condition appears in the study of optimality for totally singular extremals, i.e. when the second derivative of the Hamiltonian is identically zero along the trajectory.
A typical example is provided by the single-input affine control systems, namely, $\dot x(t)=f_0(x(t))+\alpha(t)f_1(x(t)),$ where $\alpha(t)\in\R,$ and $f_0,f_1$ are smooth vector fields.
In this case the generalized Legendre condition writes
\[\langle p(\cdot)[f_1,[f_0,f_1]]x(\cdot)\rangle>0\]
where $[\cdot,\cdot]$ is the Lie bracket of vector fields.
Our control system~\eqref{eq:control_syst_alpha} belongs to this class of problems, and straightforward computations show that for $x(t)=e^{\lambda^*t}e_1^*$ and $p(t)=e^{-\lambda^*t}\phi_1^*,$ we have $p(\cdot)[F,[G,F]]x(\cdot)=\phi^*F(G+\alpha^*F-\lambda^*I)Fe_1^*.$
The generalized Legendre condition ensures that the quadratic form $Q|_{\text{Ker}D_{\alpha^*}F_T}$ is definite negative.
Then it allows to deduce that the trajectory $x(\cdot)$ is locally optimal for short final times $T$ in the $C^0$ topology (we refer to~\cite{AgrachevSachkov}
 and~\cite{BonnardCaillauTrelat}).

Here we proved the negativity of $Q$ on $L^\infty_0$ instead of $\text{Ker}D_{\alpha^*}F_T$ in order to make clearer the link with the computation of the second derivative of Floquet eigenvalue.
It follows from the previous section that the second derivative of $\lambda_F$ is positive for periodic controls $\cos(\omega t)$ when $\omega$ is large if the condition
\begin{equation}\label{floquet_condition}\sum_{i=2}^n(\lambda_1^*-\lambda_i^*)(\phi_1^*Fe_i^*)(\phi_i^*Fe_1^*)>0.\end{equation}
Considering $T$ small and $\omega=k\frac{2\pi}{T}$ with $k\geq1,$ we have that $\cos(\omega t)\in L^\infty_0$ and $\omega\to+\infty$ when $k\to+\infty.$
The following proposition makes concordant conditions~\eqref{legendre_condition} and~\eqref{floquet_condition}.
\begin{proposition}\label{prop:equivalence_Legendre_Floquet}
We have
\[\phi^*F(G+\alpha^*F-\lambda^*I)Fe_1^*=-\sum_{i=2}^n(\lambda_1^*-\lambda_i^*)(\phi_1^*Fe_i^*)(\phi_i^*Fe_1^*).\]
\end{proposition}
This relation is instructive since it emphazises the relation between a condition for optimality for small times and a condition for optimality with high frequences.
\if{
\begin{proof}
We decompose $Fe_1^*$ on the basis $(e_i^*)_{1\leq i\leq n}$
\[Fe_1^*=\sum_{i=1}^n(\phi_i^*Fe_1^*)\,e_i^*.\]
Then we get
\[\lambda^*\phi_1^*F^2e_1^*=\lambda^*\sum_{i=1}^n(\phi_i^*Fe_1^*)(\phi_1^*Fe_i^*)\]
and
\[\phi_1^*F(G+\alpha^*F)Fe_1^*=\sum_{i=1}^n\lambda_i^*(\phi_i^*Fe_1^*)(\phi_1^*Fe_i^*)\]
\end{proof}}\fi
}\fi
\subsection{Lack of controllability/coercivity and the ergodic set}
Classical arguments for proving ergodicity results such as Theorem \ref{thm:optimal growth} or Corollary \ref{cor:ergodicity} rely on short time dynamics of the system. This is the case for instance of the ergodicity result in Capuzzo-Dolcetta and Lions \cite{CDL},
and of the Weak KAM Theorem of Fathi~\cite{Fathi}. In the former the authors assume a uniform controllability condition,  
\begin{equation*} (\exists r>0)\quad  (\forall y \in  \S )\quad  B(0,r) \subset \overline{\mathrm{co}}\{ b(y,m)\; |\; m \in \M \}\,,  
\end{equation*}
The latter relies on a regularizing property of the Lax-Oleinik semi-group which holds true for Tonelli Lagrangians. Both cases imply that the Hamiltonian is coercive, {\em i.e.} $\lim_{|p|\to +\infty} H(p,y) = +\infty$, a property which is not satisfied in our case. 
One noticeable exception can be found in \cite[Section VII.1.2]{Bardi-Capuzzo}, where the controllability condition is replaced by a dissipativity condition which is somehow similar to our uniform contraction estimate in Lemma \ref{lem:contraction}. 

The lack of controllability is clear from Lemma \ref{lem:stable cone}, where some strict subsets of the simplex are \vc{positively} invariant by all the flows. In a couple of papers, Arisawa made clear the equivalence between ergodicity (in the sense of Corollary \ref{cor:ergodicity}) and the existence of a so-called ergodic set when controllability  is lacking. The ergodic set satisfies the following properties: it is non empty, closed and \vc{positively} invariant by the flows; it is attractant; it is controllable. We refer to \cite{Arisawa1,Arisawa2} for the precise meaning of this statement, and to Section \ref{sec:example} for illustrations of ergodic sets in low-dimensional examples. 

\section{Illustration in dimension $2$}
We first illustrate our results on a simple two-dimensional
example.
\begin{figure}[t!]
\begin{center}
\includegraphics[width = 0.8\linewidth]{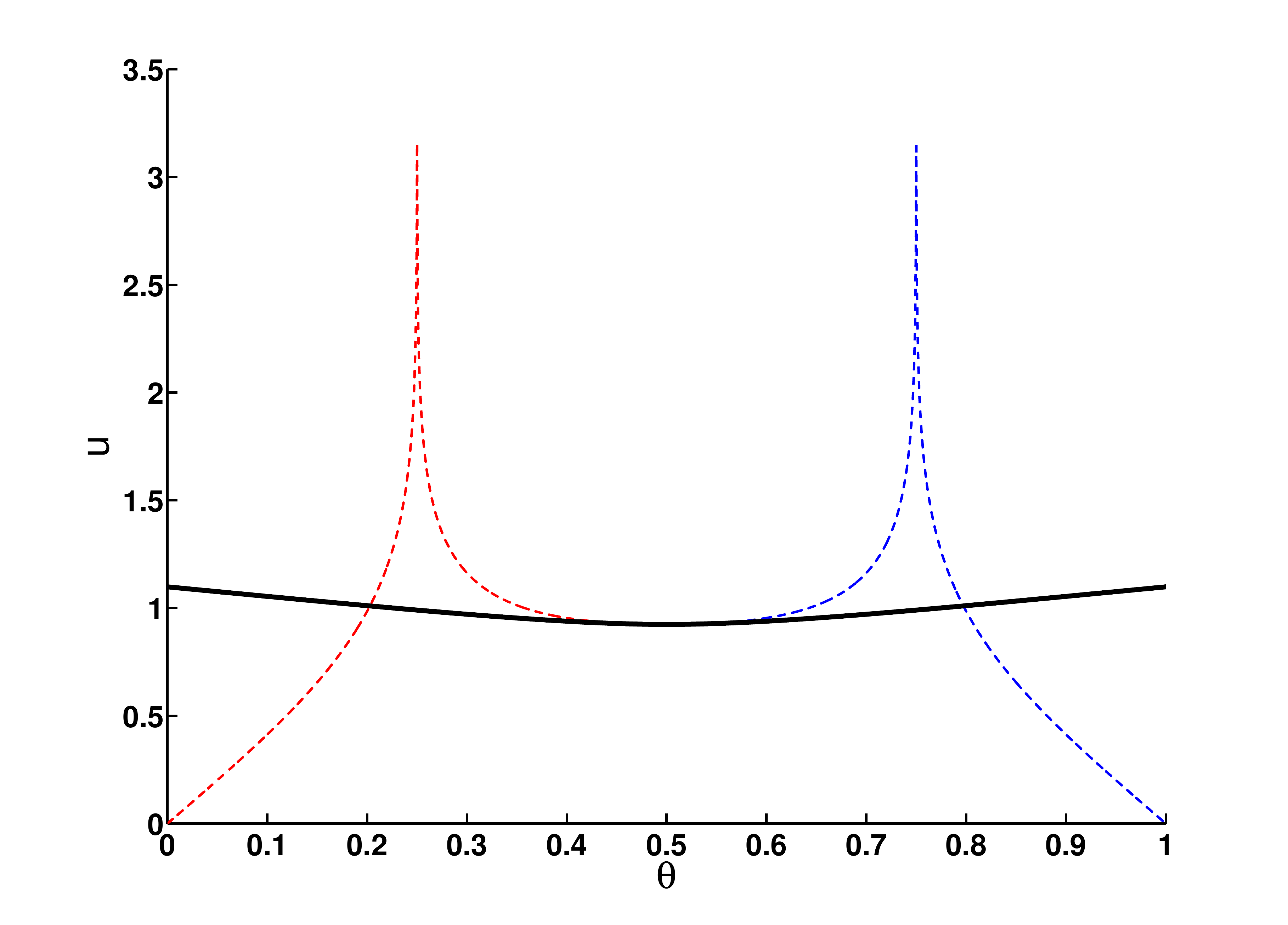}
\caption{Solution of the Hamilton-Jacobi equation \eqref{eq:ergodic HJ} for the critical value $\lambda = \frac12$. The two branches (dotted lines, red and blue) correspond to the two possible choices $\alpha = a$, resp. $\alpha = 1-a$. The bold line corresponds to the only possible combination of the two branches which gives a 
bounded viscosity solution.}
\end{center}
\end{figure}
Let $a\in (0,\frac12)$. We consider the one-parameter family  of matrices 
\[ \mathcal M = \left\{ \begin{pmatrix} 0 & 1 - \alpha \\ \alpha& 0 \end{pmatrix}\, , \quad \alpha\in [a,1 - a] \right\}\,.\]
The dominant eigenvalue is $\sqrt{\alpha(1 - \alpha)}$, with a maximum attained at $\alpha = \frac12$. 
We identify  $[0,1]$ and $\S$ under the parameterization
\[ \S = \left\{ \begin{pmatrix} 1 - \theta \\ \theta \end{pmatrix}\, , \quad \theta\in [0,1] \right\}\, . \]
We have 
$\l(\theta,\alpha) = \alpha (1 - \theta) + (1 - \alpha) \theta$, and 
$b(\theta,\alpha) = \alpha (1-\theta)^2 -  (1-\alpha) \theta^2$.
We look for a solution $(\lambda,\mathring u)$ to the Hamilton-Jacobi equation \eqref{eq:ergodic HJ}. 
For computing $\mathring u$ we have to combine two branches, corresponding to the choice  $\alpha = a$ or $\alpha = 1-a$. We realize that the first branch has a vertical asymptote at $\theta = z_a$ since $b(\cdot, a)$ vanishes at this point, whereas the second branch has a vertical asymptote at $\theta = z_{1 -a}$. Therefore there is only one possible combination which gives a bounded viscosity solution on $\S$. 

In this example the ergodic set is the segment $[z_a, z_{1-a}]$.

\section{Application to the optimization of growth-fragmentation processes in dimension 3}
\label{sec:example}

We consider a toy model for a stage-structured linear polymerization-fragmentation process. It is inspired from a nonlinear discrete polymerization-fragmentation introduced in \cite{Masel} for the dynamics of prion proliferation. We do not take into account nonlinear saturation effects, and we further reduce the size of the system to $n=3$. Polymers can have three states relative to their lengths: small (monomers), medium (oligomers), large (polymers). We denote by $x_i$, $i = 1,2,3$ the density of polymers in each compartment.  Transition rates due to growth in size of polymers from smaller to larger compartments (polymerization) are denoted by $\tau_i$, $i = 1,2$. Transition rates due to fragmentation from larger to smaller compartments are denoted by $\beta_i$, $i = 2,3$. The corresponding matrices are 
\beq\label{eq:example}\!\!G = \begin{pmatrix}
-\tau_{1}& 0 & 0 \\
\tau_{1} &-\tau_{2}& 0 \\
0 & \tau_{2} & 0 
\end{pmatrix} \,\text{and}\,
F = \begin{pmatrix}
0 & 2\beta_2 & \beta_3 \\
0 & -\beta_2 & \beta_3 \\
0 & 0 & -\beta_3
\end{pmatrix}\,.\eeq
Denoting by $q = (1\; 2 \; 3)^T$ the vector of relative sizes of polymers, we have the following properties: $\un^T G = 0$ (conservation of the number of polymers by growth) and $q^T F = 0$ (conservation of the total size of polymers by fragmentation). 

Optimal control issues come up in the development of efficient diagnosis tools for early detection of prion diseases from blood samples.  
The protocol PMCA (Protein Misfolding Cyclic Amplification) has been introduced by Soto  and co-authors \cite{Soto,Castilla} as very powerful method to achieve this goal. It aims at quickly generating {\em in vitro} detectable quantities of PrPsc being given minute quantities of it.  
PMCA consists in successive switching between incubation phases (where aggregates are expected to grow following a seeding-nucleation scenario alimented by purified PrPc) and sonication phases (where breaking of polymers is expected to increase the number of nucleation sites). This clear distinction between two phases with a control parameter which is the intensity of sonication makes the framework of \eqref{eq:ODE} very weel adapted to model PMCA.

The minimal model for PMCA goes as follows: introduce $\alpha:[0,t]\to [a,A]$ the intensity of sonication ({\em i.e.} fragmentation). The  goal is to maximize the total size of polymers $\la q, x_\alpha(t)\ra$ following the linear growth-fragmentation process:
\begin{equation} \begin{cases}
\dot x_\alpha(s) = (G + \alpha(s)F) x_\alpha(s),\\
x_\alpha(0) = x\,.
\end{cases} 
\label{eq:ODE PMCA}
\end{equation}
A generalization of this model, which includes an incidence of the sonication on the growth process, is investigated in~\cite{CoronGabrielShang}.
For problem~\eqref{eq:ODE PMCA}, corollary \ref{cor:ergodicity} implies that $\la q, x_\alpha(t)\ra$ has an exponential growth with exponent $\lambda(\M)$. When $\tau_1 = \tau_2$ it is clearly better to sonicate as much as possible ($\alpha(s)\equiv A$) because smaller monomers are equally efficient at growing in size than intermediate oligomers, but they are more numerous for a given size. However there are some biological evidence that polymerization rate is size-dependent: polymerization of intermediate aggregates have been postulated to be the most efficient \cite{Silveira} (see also \cite{CL1} for a continuous PDE model and a discussion of this phenomenon). Mathematically speaking we have a precise description of the variations of  $\lambda(\alpha)$ as stated in  the following Proposition.
\begin{proposition}\label{prop:n=3}
The Perron eigenvalue $\lambda(\alpha)$ of $G + \alpha F$ reaches a maximum value for some $\alpha^* \in (0,+\infty)$ if and only if $\tau_2>2\tau_1$. 
Furthermore we face  the following alternative:
\begin{itemize}
\item either $\tau_2\leq 2\tau_1$ and $\lb(\al)$ increases from $0$ to $\tau_1,$
\item or $\tau_2>2\tau_1$ and $\lb(\al)$ increases from $0$ to $\lb(\al^*)$ and then decreases from $\lb(\al^*)$ to $\tau_1$.
\end{itemize}
\end{proposition}
We  refer to~\cite{calvezgabriel}
for the details of the proof of Proposition~\ref{prop:n=3}.

\begin{figure}
\begin{center}
\includegraphics[width = 0.5\linewidth]{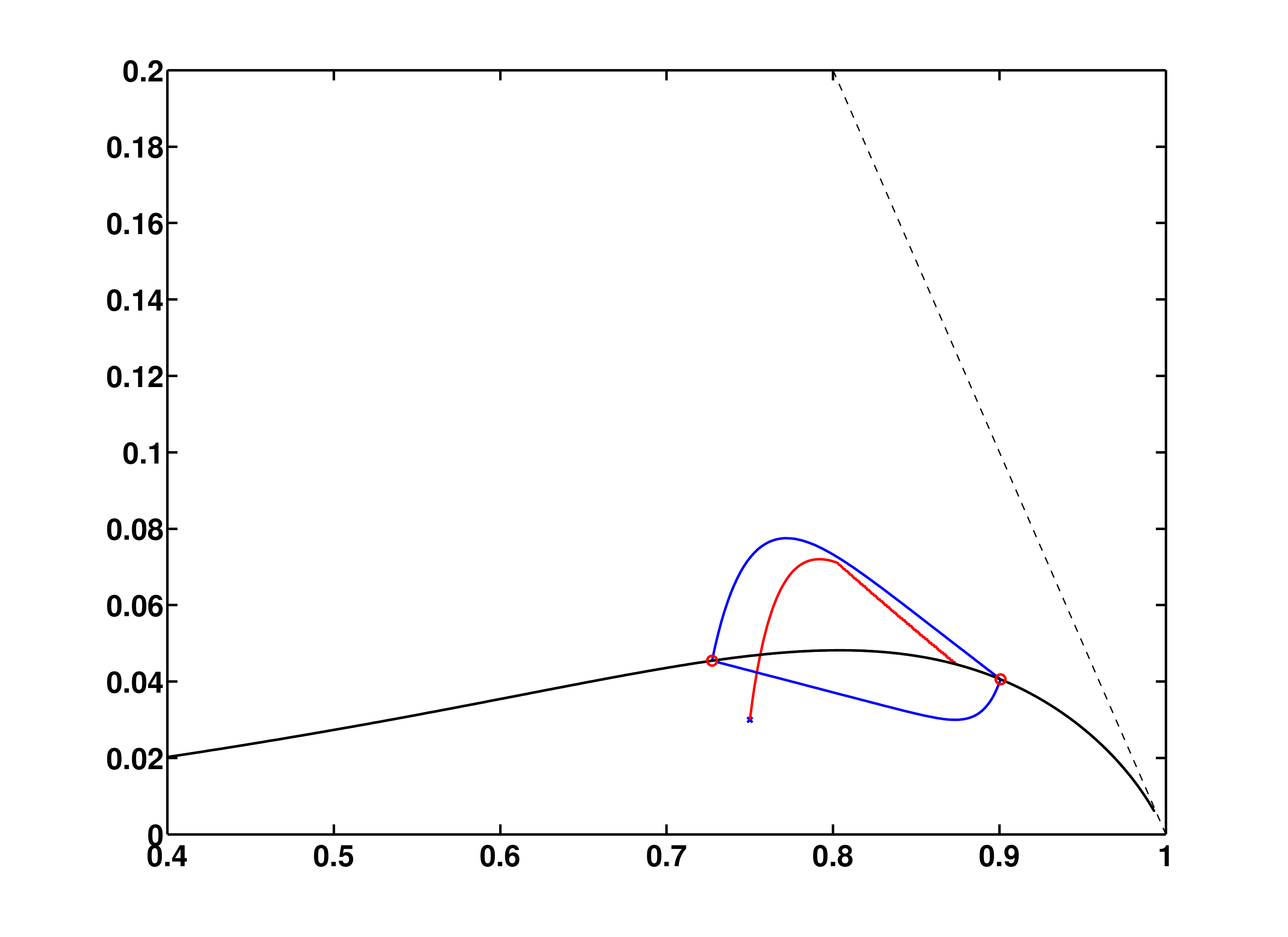}\includegraphics[width = 0.5\linewidth]{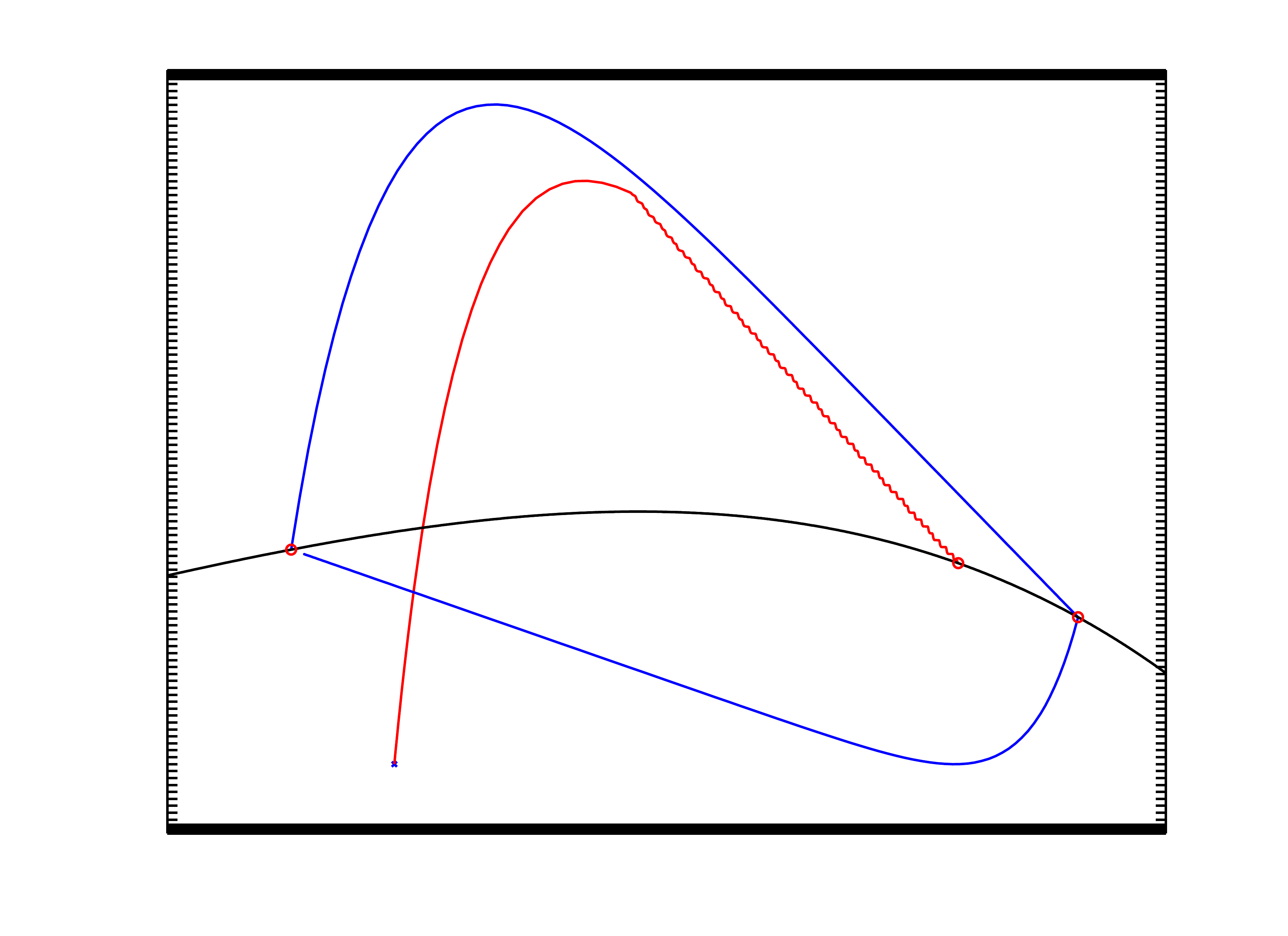}
\caption{Illustration of convergence of optimal trajectories towards the optimal Perron eigenvector for the three-dimensional growth-fragmentation process \eqref{eq:example} with typical parameters. We have plotted the curve of Perron eigenvectors (black  line), the boundary of the ergodic set (blue line), and an arbitrary optimal trajectory (red line). The optimal trajectory enters the ergodic set and then converges towards a limit cycle. Right picture is a zoom of the left one. The numerical grid is plotted on the axes: the space step is $\Delta x = 5\text{\sc \footnotesize E}{-4}$ to fall much below the width of the ergodic set. Parameters are $\tau_1 = 2\text{\sc \footnotesize E}-2, \tau_2 = 1, \beta = 4\text{\sc \footnotesize E}-2$, and $a = 2, A = 8$.} \label{fig:PMCA}
\end{center}
\end{figure}

Qualitative analysis of the optimal control and associated optimal trajectories rely on the description of relevant sets in the simplex. The ergodic set introduced by Arisawa~\cite{Arisawa1,Arisawa2}  can be characterized as the set enclosed by two remarkable trajectories: each starting from one  of the two extremal Perron eigenvectors (resp. $e_1(a)$ and $e_1(A)$) and evolving with constant control (resp. $A$ and $a$).  This set is of particular interest since it attracts all trajectories, not necessarily optimal ones. However it does not give any insight about the fate of optimal trajectories inside the ergodic set. 

So far we have only access to local second-order conditions \eqref{eq:secondPerron} to test the optimality of the best constant control. Numerical tests suggest that we always have  $\lambda(\M) = \max_m  \lambda(m)$ in the case of \eqref{eq:example}, that is to say the optimal trajectory converges towards the optimal Perron eigenvector in the simplex. This is confirmed by finite-volume numerical simulations of the ergodic Hamilton-Jacobi equation \eqref{eq:ergodic HJ:intro}, see Figure \ref{fig:PMCA}. We computed the optimal control $\alpha^*$ as a function of the position in the simplex. We observed a clear separation between two connected regions of the simplex (result not shown), corresponding to the extremal choices $\alpha = a$ or $\alpha = A$. This gives an optimal vector field which drives optimal trajectories. Outcomes of the numerical simulations  are consistent with the presumable stability of the best constant control against periodic perturbations.

\begin{figure}
\begin{center}
\includegraphics[width = 0.5\linewidth]{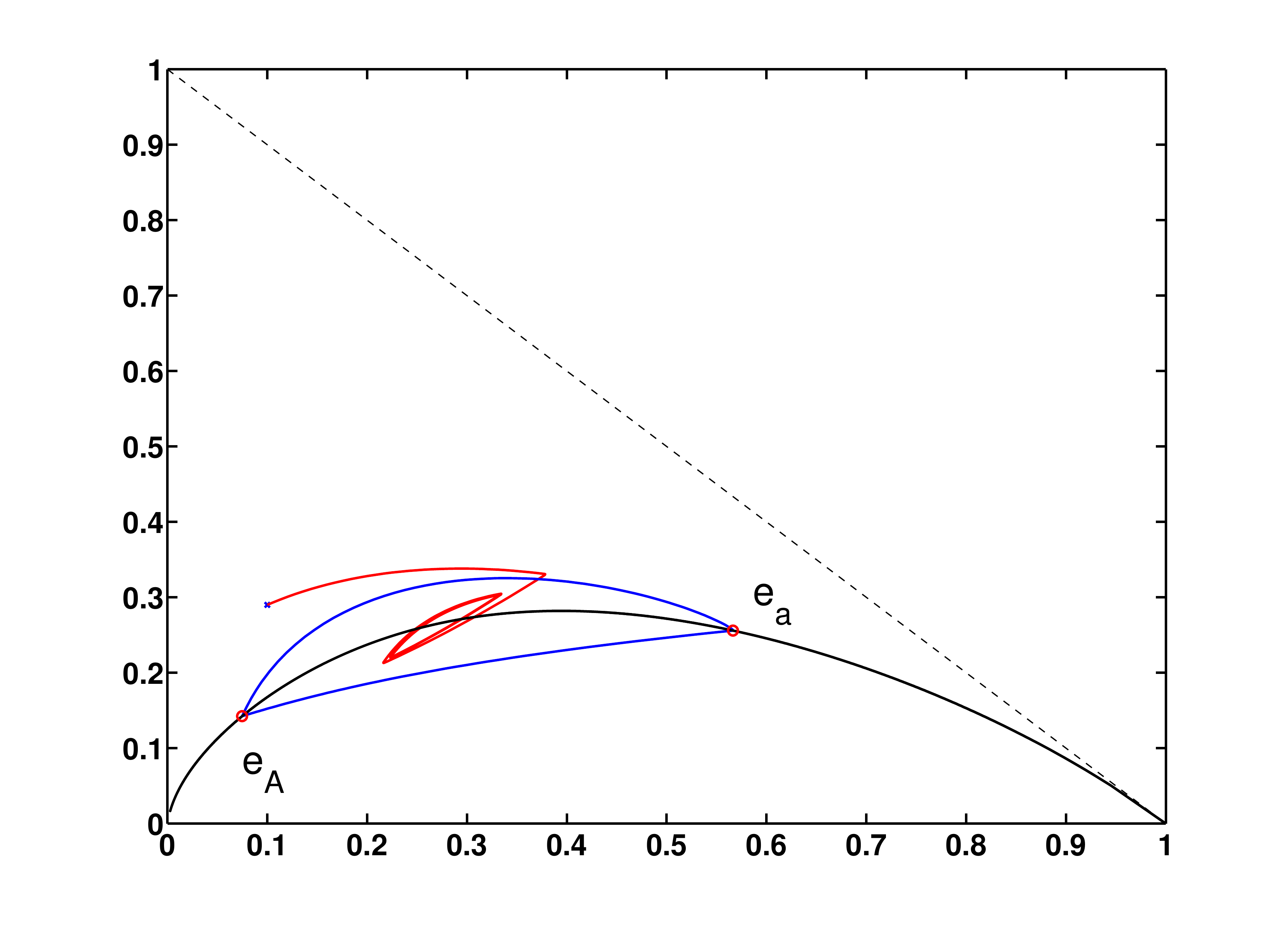}\includegraphics[width = 0.5\linewidth]{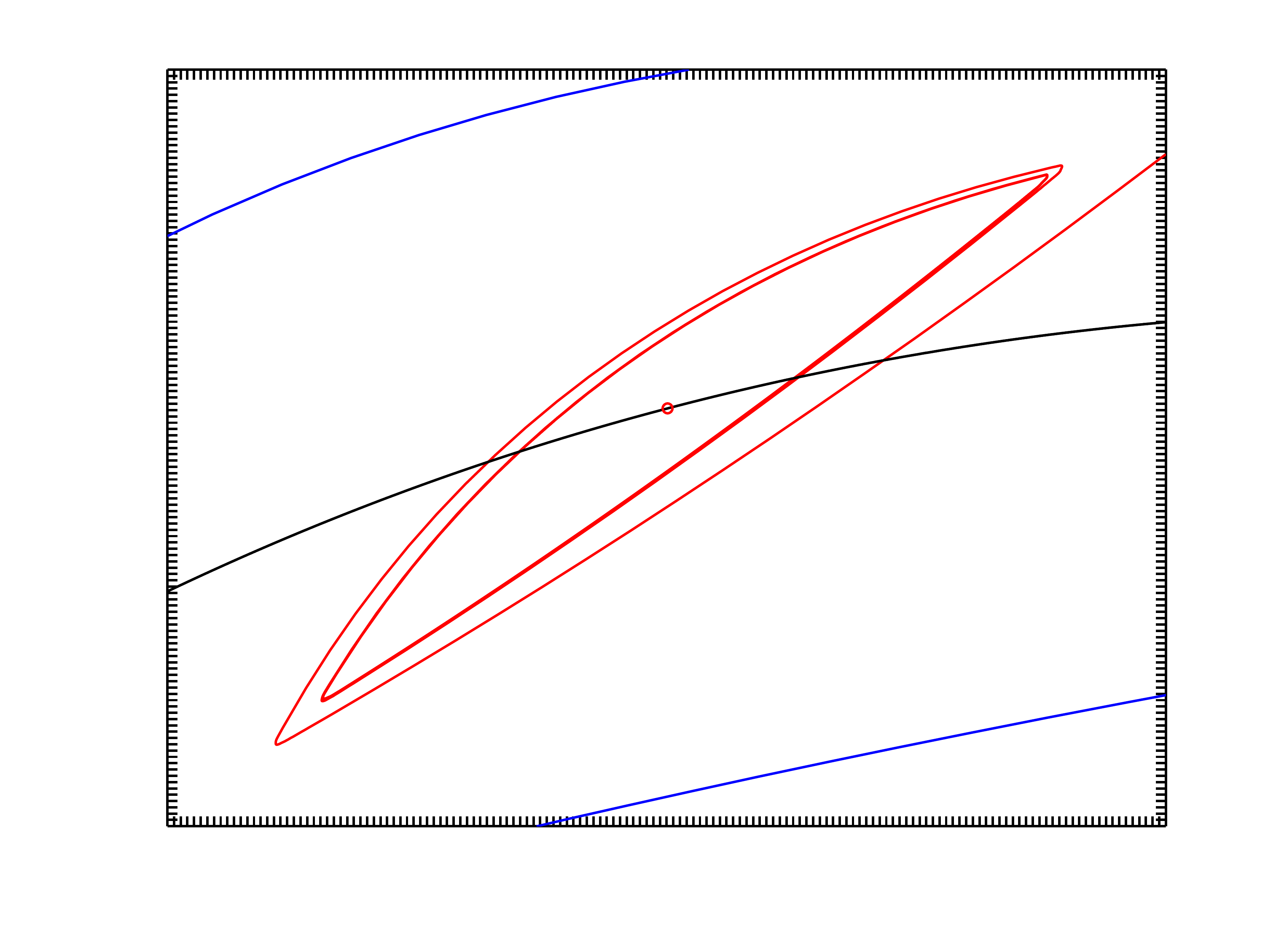}
\caption{Illustration of the optimal limit cycle for the numerical example \eqref{eq:example:new}. We have plotted the curve of Perron eigenvectors (black  line), the boundary of the ergodic set (blue line), and an arbitrary optimal trajectory (red line). The optimal trajectory enters the ergodic set and then converges towards a limit cycle. Right picture is a zoom of the left one. The numerical grid is plotted on the axes: the space step is $\Delta x =1\text{\sc \footnotesize E}{-3}$ to fall much below the width of the limit cycle.} \label{fig:limit-cycle}
\end{center}
\end{figure}

\subsection*{A three-dimensional example with an optimal limit cycle}
Proposition \ref{th:dim n=2} rules out the existence of optimal limit cycles in dimension $n=2$. Although the previous example of the three-dimensional growth-fragmention process did not exhibit limit cycles apparently, we were able to find another three-dimensional example by testing random choices of matrices with respect to the stability criterion \eqref{eq:secondPerron}:
\begin{align}\nonumber
G &=  \begin{pmatrix}
0 & 0.245 & 0.007 \\
0 & 0 & 0.141 \\
0 &0 &0 
\end{pmatrix} \\ 
F &= \begin{pmatrix}
-0.245 & 0 & 0 \\
0.272 & -0.499 & 0 \\
0.645 & 0.026 & -0.035
\end{pmatrix}\, .
\label{eq:example:new}
\end{align}
We assume as in the previous example that the control $\alpha$ takes values in $[a,A]$. The maximal Perron eigenvalue is obtained for $\alpha^*\approx 0.415$. The stability criterion \eqref{eq:secondPerron} has been checked numerically: the optimal constant control is not stable with respect to periodic perturbations at high frequency. Therefore we expect limit cycles to attract optimal trajectories in the simplex in the spirit of the Poincar\'e-Bendixson theory. This has been checked using finite-volume numerical simulations of the ergodic Hamilton-Jacobi equation \eqref{eq:ergodic HJ:intro}, see Figure \ref{fig:limit-cycle}. It is worth mentioning that a similar counter-example has been proposed in \cite{fainshil_stability_2009} to answer a question raised in \cite{gurvits_stability_2007}. Our quantitative approach based on second-order conditions provides another example. Furthermore it illustrates the rich possible dynamics of optimal trajectories. The connexion with the Poincar\'e-Bendixson theory seems appealing. However, proving that limit cycles are the only alternative to pointwise convergence seems out of reach at the moment, due to the complexity of the Poincar\'e-Bendixson theory in the case of discontinuous vector fields \cite{de_carvalho_poincare-bendixson_2013}.

\bibliographystyle{abbrv}
\bibliography{HJB}

\end{document}